\documentclass[11pt]{amsart}
\usepackage{cite}
\usepackage{wrapfig, lipsum,booktabs}
\usepackage{graphicx}
\usepackage{amssymb}
\usepackage{epstopdf}
\usepackage{enumitem} 
\usepackage{verbatim}
\usepackage{bm}
\usepackage{multicol}
\usepackage{multirow}
\usepackage{subfigure}
\usepackage{fancyhdr} 
\usepackage{float}
\usepackage{stmaryrd}
\usepackage{color}
\usepackage{soul}
\usepackage{tikz}
\usepackage{todonotes}
\usetikzlibrary{patterns}
\usepackage{pgfplots}
\usepackage{cancel}
\usepackage{amsmath}
\usepackage{mathtools}
\usepackage{calrsfs}
\DeclareMathAlphabet{\pazocal}{OMS}{zplm}{m}{n}
\newtheorem{theorem}{Theorem}[section]
\newtheorem{lemma}{Lemma}[section]
\newtheorem{remark}{Remark}[section]

\usepackage{multirow}
\usepackage{amsmath,amssymb,eucal}
\usepackage{graphicx,subfigure,epsfig}
\usepackage{psfrag}
\usepackage{url}
\usepackage[top=1in, bottom=1.in, left=1in, right=1in]{geometry}

\newcommand{\bld}[1]{\hbox{\boldmath$#1$}}    
\newcommand{\Th}{\mathcal{T}_h}

\newcommand{\Eh}{\mathcal{E}_h}

\setulcolor{red}
\newcommand{\jmp}[1]{[\![#1 ]\!]}
\newcommand{\bldu}{\bld{u}}
\newcommand{\bldv}{\bld{v}}
\newcommand{\blduhat}{\widehat{\bld{u}}}
\newcommand{\bldvhat}{\widehat{\bld{v}}}
\newcommand{\lineu}{\underline{\bld{u}}}
\newcommand{\linev}{\underline{\bld{v}}}
\newcommand{\bldV}{\bld{V}}

\newcommand{\matrixU}{\mathbf{U}}
\newcommand{\matrixV}{\mathbf{V}}
\newcommand{\matrixlineU}{\mathbf{\underline{U}}}
\newcommand{\matrixlineV}{\mathbf{\underline{V}}}
\newcommand{\matrixP}{\mathbf{P}}
\newcommand{\matrixQ}{\mathbf{Q}}
\newcommand{\matrixUhat}{\mathbf{\widehat{U}}}

\newcommand{\Vhat}{\widehat{V}}

\newcommand{\bldVhat}{\widehat{\bld{V}}}
\newcommand{\lineV}{\underline{\bld{V}}}
\newcommand{\poly}{\mathcal{P}}
\newcommand{\matrixA}{\mathbf{A}}
\newcommand{\matrixB}{\mathbf{B}}
\newcommand{\matrixC}{\mathbf{C}}
\newcommand{\matrixD}{\mathbf{D}}
\newcommand{\trans}{\mathrm{T}}

\makeatletter
\renewcommand*\env@matrix[1][\arraystretch]{%
	\edef\arraystretch{#1}%
	\hskip -\arraycolsep
	\let\@ifnextchar\new@ifnextchar
	\array{*\c@MaxMatrixCols c}}
\makeatother

\newcommand\blue[1]{{{#1}}}

\newcommand\revise[1]{{#1}}
\begin{document}
\title[SADDLE-HDG]{Uniform block-diagonal preconditioners for 
divergence-conforming HDG Methods for the generalized Stokes equations and the linear elasticity equations}

\author{Guosheng Fu}
\address{Department of Applied and Computational Mathematics and 
Statistics, University of Notre Dame, USA.}
\email{gfu@nd.edu}
\thanks{We gratefully acknowledge the partial support of this work
	by the U.S. National Science Foundation through grant DMS-2012031.}
\author{Wenzheng Kuang}
\address{Department of Applied and Computational Mathematics and 
	Statistics, University of Notre Dame, USA.}
\email{wkuang1@nd.edu}

\begin{abstract}
	We propose a uniform block-diagonal preconditioner for condensed $H$(div)-conforming HDG schemes for 
	parameter-dependent saddle point problems, including the generalized Stokes
	equations and the linear elasticity equations. 
	An optimal preconditioner is obtained for the stiffness matrix on the
	global velocity/displacement space via the auxiliary space preconditioning (ASP) technique \cite{Xu96}. 
	\revise{A spectrally equivalent approximation to the Schur complement on the element-wise constant pressure space is also constructed, and an explicit computable exact inverse is obtained via the Woodbury matrix identity.} Finally, the numerical results verify the robustness of our proposed preconditioner with respect to model parameters and mesh size.
\end{abstract}
\maketitle

\section{Introduction}
\label{sec:intro}
Since their first introduction for second order elliptic problems about a decade ago \cite{CockburnGopalakrishnanLazarov09}, hybridizable discontinuous Galerkin (HDG) schemes have been developed and successfully applied to various partial differential equations (PDEs) in computational fluid dynamics \cite{cockburn2018discontinuous,qiu2016superconvergent}, wave propagation \cite{cockburn2016hdg}, and continuum mechanics \cite{nguyen2012hybridizable,fu2021locking}. One key advantage of HDG schemes over discontinuous Galerkin (DG) schemes is that they can be statically condensed into a reduced linear system with increased sparsity, resulting in a significant decrease in the matrix size and computing cost \cite{NguyenPeraireCockburnConvDiff09, Cockburn16}.

\blue{
However, the challenge of constructing optimal and robust solvers and preconditioners for condensed HDG schemes has not been fully addressed, and techniques such as multigrid and domain decomposition methods have been mainly explored.
Cockburn et al. introduced a V-cycle geometric multigrid method for HDG schemes for the elliptic equations\cite{CockburnDubois14}, where a continuous element-wise linear function space is used at the second level, and then a standard conforming multigrid method starts from there. A similar idea was employed for the Helmholtz equations \cite{chen2014robust}, the shallow water equations\cite{betteridge2021multigrid}, and extended to an $hp$-multigrid in parallel manner in \cite{fabien2019manycore}. We note that Lu et al. recently reported a homogeneous multigrid for HDG schemes for the elliptic equations where HDG discretization is used in each level \cite{lu2020hmg}.
Standard $p$-version domain decomposition methods were first analyzed by Sch\"oberl and Lehrenfeld for statically condensed systems of high order HDG schemes for the elliptic equations where each element is treated as a sub-domain \cite{schoberl2013domain}. Schwarz type methods and balancing domain decomposition with constraints (BDDC) algorithms for HDG schemes have been further investigated for the elliptic equations\cite{gander2015analysis ,gander2018analysis}, the incompressible Stokes equations\cite{barrenechea2019hybrid, tu2020analysis}, the Maxwell equations\cite{li2014hybridizable, he2016optimized} and the hyperbolic equations\cite{muralikrishnan2017ihdg, muralikrishnan2018improved}.


In this paper, we focus on a uniform preconditioner for the divergence-conforming HDG schemes for the parameter-dependent saddle point problems which arise when dealing with the mixed finite element formulation of the generalized Stokes equations, the linear elasticity equations, and the Brinkman equations. $H$(div)-conforming HDG discretizations for these equations have been developed in previous works \cite{LehrenfeldSchoberl16, fu2021locking, fu2019parameter} for the pressure-robust and mass-conserving properties but without efficient solvers. A wealth of literature is devoted to solving saddle point problems, and we refer to \cite{benzi2005numerical} for a comprehensive review of methods, including block-diagonal preconditioners, domain decomposition methods, multilevel methods, and so on. We note that recently block-factorization preconditioners were also proposed for a different HDG scheme for the incompressible Stokes equations\cite{rhebergen2018preconditioning, rhebergen2021preconditioning}. However, robust preconditioners for the $H$(div)-conforming HDG schemes for the generalized parameter-dependent saddle point problems have not been addressed in the literature yet.


 \revise{Block-diagonal preconditioners have been well-established for the stabilized Stokes equations, the linear elasticity equations\cite{wathen1993fast, silvester1994fast}, and the generalized Stokes equations\cite{bramble1997iterative, peters2005fast}. For the generalized Stokes equations, the Schur complement preconditioner that is robust with respect to mesh size and time step was first proposed in \cite{cahouet1988some}, then theoretical proved in the finite element setting in \cite{bramble1997iterative} and in continuous setting in \cite{kobelkov2000effective}.} 
 We also refer to the surveys in\cite{mardal2011preconditioning, pestana2015natural} for the generalized framework and analysis of block-diagonal preconditioners for the saddle point problems.}
 Here, we present a uniform block-diagonal preconditioner for the condensed $H$(div)-conforming HDG schemes for the parameter-dependent saddle point problems, including the generalized Stokes equation and the linear elasticity. The key idea is to find robust approximations for the matrix inverses of the symmetric positive definite (SPD) stiffness matrix on the global velocity/displacement space and the (negative) Schur complement on the element-wise constant pressure space in the statically condensed system. For the stiffness matrix on the global velocity/displacement space, we continue the work in our previous study on the reaction-diffusion equations \cite{fu2021uniform} and construct an optimal auxiliary space preconditioner (ASP) based on the theory proposed by Xu in \cite{Xu96}. For the Schur complement on the element-wise constant pressure space, we mainly borrow ideas from \cite{Mardal04, Olshanskii06}. \revise{We define a parameter-dependent norm on the element-wise constant pressure space, then we prove the spectral equivalence between the newly defined norm and the one induced by the (negative) Schur complement, which is robust with respect to model parameters and mesh size. It needs to be pointed out that elliptic regularity is assumed for the domain here. Next, through variational analysis, we construct an explicit matrix formulation corresponding to the definition of the newly defined norm, and naturally conclude that its inverse is robust Schur complement preconditioner. The efficient computation of this Schur complement preconditioner is realized via the Woodbury matrix identity\cite{higham2002accuracy}.} The numerical experiments verify the robustness of the preconditioner.

The rest of the paper is organized as follows. In Section \ref{sec:prelimi}, we introduce the $H$(div)-conforming HDG scheme for the generalized parameter-dependent saddle point problems and express the static condensation process in matrix formulation. In Section \ref{sec:precondition}, block-diagonal preconditioners that are robust with respect to model parameters and mesh size are constructed for the statically condensed system. Numerical examples based on the generalized Stokes equations and the linear elasticity equations are then presented in Section \ref{sec:num} to verify the robustness of our proposed preconditioner, and we conclude in Section \ref{sec:conclude}.

\section{$H$(div)-conforming HDG for the Model Problem}
\label{sec:prelimi}
\subsection{Notations and finite element spaces}
We assume the domain $\Omega\in\mathbb{R}^{d}$, with $d=2,3$, to be convex polygonal/polyhedral. Let $\mathcal{T}_h$ be a shape-regular, quasi-uniform, conforming simplicial triangulation of the domain $\Omega$. 
For each element $K\in\Th$ we denote by $h_K$ its diameter, and by $h$ the maximum diameter on the mesh $\Th$. We denote $\Eh$ as the set of facets of the mesh $\Th$, which we also refer to as the \textit{mesh skeletons}. 
We split $\Eh$ into boundary facets $\Eh^\partial:=\{F\in\Eh:F\subset\partial\Omega\}$, and interior facets $\Eh^o:=\Eh\backslash\Eh^\partial$. 
Given any facet $F\in\Eh$ with normal direction $\bld{n}$, we denote $\mathsf{tang}(\bld w):=\bld{w}-(\bld{w}\cdot\bld{n})\bld{n}$ as the tangential component and $\jmp{\bld w}$ as the jump on two adjacent element of a vector field $\bld{w}$. 
Given a simplex $S\subset\mathbb{R}^d$, with $d=1,2,3$, we denote $\mathcal{P}^m(S)$ ,$m\ge 0$, as the space of polynomials of degree at most $m$ on $S$. For any function in $L^2(S)$, we denote $(\cdot,\cdot)_S$ as the $L^2$ inner product if $S\in\Th$, or $\langle\cdot,\cdot\rangle_S$ if $S\in\Eh$, we denote $\|\cdot\|_S$ as the corresponding $L^2$ norm on the simplex $S$. For any functions in $L^2(\Omega)$, we denote $(\cdot,\cdot)_{\Th}:=\sum_{K\in\Th}(\cdot,\cdot)_K$ as the discrete $L^2$ inner product on the whole domain $\Omega$ and $\|\cdot\|_{\Th}:=(\cdot,\cdot)_{\Th}^{1/2}$ as the corresponding norm. For any functions in $L^2(\Eh)$, We denote $\langle\cdot,\cdot\rangle_{\Eh}:=\sum_{F\in\Eh}\langle\cdot,\cdot\rangle_F$ as the discrete $L^2$ inner product on the mesh skeletons $\Eh$ and $\|\cdot\|_{\Eh}:=(\cdot,\cdot)_{\Eh}^{1/2}$ as the corresponding norm. For $A,B\in\mathbb{R}^+$, we write $A\lesssim B$ to indicate there exists a positive constant $C$ such that $A\le CB$, with $C$ only dependent on shape regularity of the mesh $\Th$ and the polynomial degree of the finite element spaces. Furthermore, we denote $A\simeq B$ when $A\lesssim B$ and $B\lesssim A$.

The following finite element spaces are used to construct the divergence-conforming HDG scheme for the model problem:
	\begin{alignat*}{2}
		&\bldV_h^k&&:=\;
        \left\{\bld v\in H(\mathrm{div};\Omega):\;\;\bld v|_K\in[\poly^k(K)]^d,\;\forall K\in\Th\right\},
        \\
		&\bldV_{h,0}^k&&:=\;
		\left\{\bld v \in \bldV_h^k:\;\;\bld v\cdot\bld n|_F=0, \;\forall F\in\Eh^\partial\right\},
		\\
		&\bldVhat_h^k&&:=\;
		\left\{
		\begin{array}{lll}
			\bldvhat\in [L^2(\Eh)]^3:\;\;\bldvhat|_F\in[\poly^0(F)]^3\oplus\bld{x}\times[\poly^0(F)]^3, \;&\forall F\in\Eh, & \text{if $k=1$ and $d=3$,}
			\\
			\bldvhat\in [L^2(\Eh)]^d:\;\;\bldvhat|_F\in[\poly^k(F)]^d, \;\bldvhat\cdot\bld{n}|_F=0, \;&\forall F\in\Eh, & \text{else,}
		\end{array}
		\right.
		\\
		&\bldVhat_{h,0}^k&&:=\;
		\left\{\bldvhat\in\bldVhat_h^k:\;\;\mathsf{tang}(\bldvhat)|_F=0, \;\forall F\in\Eh^{\partial}\right\},
		\\
		&Q_h^k&&:=\;
		\left\{q\in L^2(\Omega):\;\;q|_K\in\poly^k(K), \;\forall K\in\Th\right\},\\
		&Q_{h,0}^k&&:=\;
		\left\{q\in Q_h^k:\;\;\int_{\Omega}q\;\mathrm{dx}=0\right\},
	\end{alignat*}
where $k$ is the polynomial degree, and $d$ the dimension of the domain $\Omega$.  Next, we perform a hierarchical basis splitting for the Brezzi-Douglas-Marini (BDM) finite element space $\bld{V}_h^k$ as was done in \cite[Section 2.2.4]{Lehrenfeld10} to facilitate our analysis of static condensation of the $H$(div)-conforming HDG scheme:
\begin{alignat*}{3}
	&\bldV_h^k&&=
	\;\;\;\;\bldV_h^{k,\partial}\oplus\bldV_h^{k,o},
	\\
	&\bldV_h^{k,\partial}&&:=
	\bigoplus_{\begin{subarray}{c}
			F\in\Eh\\{i=1,\dots,n_F^k}
		\end{subarray}}\mathrm{span}\{\phi_F^i\}&&\oplus
	\;\;\bigoplus_{F\in\Eh}\mathrm{span}\{\phi^0_F\},
	\\
	&\bldV_h^{k,o}&&:=
	\bigoplus_{\begin{subarray}{c}
			K\in\Th\\i=1,\dots,n_K^{k,1}
	\end{subarray}}\mathrm{span}\{\phi_{K}^{i}\}&&\oplus
	\bigoplus_{\begin{subarray}{c}
			K\in\Th\\i=1,\dots,n_K^{k,2}
	\end{subarray}}\mathrm{span}\{\psi_K^i\},
\end{alignat*}
where $\bldV_h^{k,\partial}$ and $\bldV_h^{k,o}$ are \textit{global} and \textit{local} subspaces of $\bldV_h^k$; $\phi_F^0$ is the basis of the lowest order Raviart-Thomas ($\mathrm{RT0}$) space on the facet $F$; $\phi_F^i$, $i=1,\dots,n_F^k$, is the higher order basis of the divergence-free facet bubbles with normal component only supported on the facet $F$; $\phi_K^i$, $i=1,\dots,n_K^{k,1}$, is the higher order basis of the divergence-free bubbles on the element $K$ with zero normal component on $\Eh$; and $\psi_K^{i}$, $i=1,\dots n_K^{k,2}$, is the higher order basis on the element $K$ with zero normal component on $\Eh$ and nonzero divergence. The integers $n_F^k,n_K^{k,1},n_K^{k,2}$ are denoted as the number of basis functions of each corresponding group per facet/element. We also split $Q_h^k$ into element-wise constant space and its complement:
\begin{alignat*}{3}
	&Q_h^k&&=\;\;
	&&\overline{Q}_h\oplus Q_h^{k,o},\\
	&\overline{Q}_h&&:=\;\;
	&&Q_h^0,\\
	&Q_h^{k,o}&&:=\;\;
	&&\left\{q\in Q_h^k:\;\;(q,\;1)_K=0,\;\forall K\in\Th\right\}.
\end{alignat*}
Through the above space splitting, we have:
\begin{subequations}
	\label{spaceRelation}
\begin{align}
	\nabla\cdot\bldV_h^{k,\partial} &= \overline{Q}_h,\\
	\nabla\cdot\bldV_h^{k,o} &= Q_h^{k-1,o},
\end{align}
\end{subequations}
which we refer to\cite[Section 2.2.4]{Lehrenfeld10} for details.  

\subsection{Model problem and the HDG scheme}
We consider the following saddle point problem: Find $\bld u$ and $p$ such that
\begin{equation}
	\label{model}
	\left.
	\renewcommand{\arraystretch}{1.5} 
		\begin{tabular}{r r r}
			$\nabla\cdot(-2\mu\bld{D}(\bld{u}))+\tau\bld{u}+\nabla p$ & $=$ & $\bld{f}$ \\
			$\nabla\cdot\bld{u}+\frac{1}{\lambda}p$ & $=$ & $0$
		\end{tabular}
	\right\} \;\text{in}\; \Omega,\;\bld{u}|_{\partial\Omega}=\bld{0},
\end{equation}
where $\bld{D}(\bld{u}):=\frac{1}{2}(\nabla\bld{u}+\nabla\bld{u}^T)$ is the symmetric gradient,
$\mu,\lambda>0$, $\tau\ge 0$ are constant model parameters and $\bld{f}$ is the source term. We note that this model covers the linear elasticity problem where $\tau\ge0$ is the inverse of time step size, $\mu,\lambda>0$ are Lam\'e parameters, with $\bldu$ representing the displacement and $p$ the pressure. It also covers the generalized Stokes problem where $\lambda=+\infty$, $\mu>0$ is the viscosity and $\tau\ge0$ is the inverse of time step size, where $\bldu$ represents the fluid velocity and $p$ the pressure.

To simplify our analysis, we focus on this constant-coefficient problem with homogeneous Dirichlet boundary conditions, while other standard boundary conditions are covered in our numerical experiments in Section \ref{sec:num}.

The equation \eqref{model} is discretized using the symmetric interior penalty 
divergence-conforming HDG (SIP-divHDG) method with projected jumps \cite[Remark 1.2.4]{Lehrenfeld10}. To further simplify our notation, we define a compound finite element space $\lineV_{h,0}^k:=\bldV_{h,0}^k\times\bldVhat_{h,0}^{k-1}$ and its elements $\linev_h:=(\bld{v}_h,\bldvhat_h)$. Then, the weak formulation of the divergence-conforming HDG scheme with polynomial degree $k\ge 1$ is given as below: Find $(\lineu_h,p_h)\in\lineV_{h,0}^k\times Q_{h,0}^{k-1}$ such that
\begin{equation}
	\label{weakForm}
	\left.
	\renewcommand{\arraystretch}{1.5} 
	\begin{tabular}{r l}
		$a(\lineu_h,\linev_h)+b(p_h,\bld{v}_h)=$ &
		$\left(\bld{f},\bld{v}_h\right)_{\Th}$,\\
		$b(\bld{u}_h,q_h)+c(p_h,q_h)=$ & $0$,
	\end{tabular}
	\right\}\forall\;\;(\linev_h,q_h)\in\lineV_{h,0}^{k}\times Q_{h,0}^{k-1},
\end{equation}
where the bilinear forms are defined as:
\begin{alignat*}{1}
	a(\lineu_h,\linev_h)&:=
	\sum_{K\in\Th}\tau(\bldu_h,\bldv_h)_K+2\mu\text{\Large$($}(\bld{D}(\bldu_h):\bld{D}(\bldv_h))_K-\langle\bld{D}(\bldu_h)\bld{n},\mathsf{tang}(\bldv_h-\bldvhat_h)\rangle_{\partial K}
	\\
	&-\langle\bld{D}(\bldv_h)\bld{n},\mathsf{tang}(\bldu_h-\blduhat_h)\rangle_{\partial K}
	+\langle\frac{\alpha k^2}{h}P_{k-1}(\mathsf{tang}(\bldu_h-\blduhat_h)),P_{k-1}(\mathsf{tang}(\bldv_h-\bldvhat_h))\rangle_{\partial K}\text{\Large$)$},
	\\
	b(p_h,\bldu_h)&:=
	-\sum_{K\in\Th}(p_h,\nabla\cdot\bldu_h)_K,
	\\
	c(p_h,q_h)&:=
	-\frac{1}{\lambda}\sum_{K\in\Th}(p_h,q_h)_K,
\end{alignat*}
where $P_{k-1}$ denotes the $L^2$-projection onto $\bldVhat_{h,0}^{k-1}$, $\alpha>0$ is the stabilization parameter to ensure the following coercivity result:
\begin{equation*}
	\label{coercivity}
	a(\lineu_h,\lineu_h)
	\gtrsim
	\sum_{K\in\Th}\tau\|\bldu_h\|_K^2+2\mu\left(\|\bld{D}(\bldu_h)\|_K^2+\frac{1}{h}\|{\mathsf{tang}}(\bldu_h-\blduhat_h)\|_{\partial K}^2\right),
\end{equation*}
for all $\lineu_h\in\lineV_{h,0}^k$. We refer to \cite{AinsworthFu18} for a detailed discussion of the lower bound of $\alpha$. By using the Cauchy-Schwarz inequality and the inverse inequality, it is also easy to verify that on the finite element space $\lineV_{h,0}^k$,
\begin{equation*}
	a(\lineu_h,\lineu_h)\lesssim
	\sum_{K\in\Th}\tau\|\bldu_h\|_K^2+2\mu\left(\|\bld{D}(\bldu_h)\|_K^2+\frac{1}{h}\|{\mathsf{tang}}(\bldu_h-\blduhat_h)\|_{\partial K}^2\right).
\end{equation*}
We refer to \cite{fu2021locking} for more details of the coercivity and boundedness results. Therefore, $a(\lineu_h,\lineu_h)$ defines a norm on the finite element space $\lineV_{h,0}^k$:
\begin{equation}
	\label{aNorm}
	a(\lineu_h,\lineu_h)\simeq|\!|\!|\lineu_h|\!|\!|_{\ast,h}^2,
\end{equation}
where $|\!|\!|\lineu_h|\!|\!|_{\ast,h}^2:=\tau\|\bldu_h\|^2_{\Th}+2\mu|\!|\!|\lineu_h|\!|\!|_{1,h}^2$, and $|\!|\!|\lineu_h|\!|\!|_{1,h}^2:=\sum\limits_{K\in\Th}\|\bld{D}(\bldu_h)\|_K^2+\frac{1}{h}\|{\mathsf{tang}}(\bldu_h-\blduhat_h)\|_{\partial K}^2$.

\subsection{Matrix formulation and static condensation}
We apply matrix representation of the $H$(div)-conforming HDG scheme in \eqref{weakForm}  to demonstrate the static condensation process. We denote the coefficient vectors of $\bldu_h,\blduhat_h,p_h$ in their basis functions by $\matrixU,\matrixUhat,\matrixP$, the coefficient vector of $\lineu_h$ by $\matrixlineU:=[\matrixU,\matrixUhat]^T$, the vector of the linear form $\left(\bld{f},\bld{v}_h\right)_{\Th}$ by $\mathbf{F}$, the Euclidean inner product in $\mathbb{R}^n$ by $\langle\cdot,\cdot\rangle_2$. Then, we define the matrix {$\matrixA, \matrixB$ and $\matrixC$} corresponding to the bilinear forms in \eqref{weakForm} by
\begin{alignat*}{2}
	\langle\mathbf{A}\matrixlineU,\matrixlineV\rangle_2&:=
	a(\lineu_h,\linev_h),\;\;&&\forall\;\lineu_h,\linev_h\in\lineV_{h,0}^k,\\
	\langle\mathbf{B}\matrixP,\matrixV\rangle_2&:=
	b(p_h,\bldv_h),\;\;&&\forall\; p_h\in Q_{h,0}^{k-1},\bldv_h\in\bldV_{h,0}^k,\\
	\langle\mathbf{C}\matrixP,\matrixQ\rangle_2&:=
	c(p_h,q_h),\;\;&&\forall\; p_h,q_h\in Q_{h,0}^{k-1}.
\end{alignat*}
Then, we have the matrix formulation of the HDG scheme \eqref{weakForm}:
\begin{equation}
	\label{originM}
	\begin{bmatrix}[1.5]
		\mathbf{A} & \mathbf{B}\\
		\mathbf{B}^\mathrm{T}    & \mathbf{C}
	\end{bmatrix}
	\begin{bmatrix}[1.5]
		\matrixlineU \\
		\matrixP
	\end{bmatrix}
	=
	\begin{bmatrix}[1.5]
		\mathbf{F} \\
		\mathbf{0}
	\end{bmatrix}.
\end{equation}
\revise{By the definition of the corresponding bilinear forms, matrix $\matrixA$ is SPD, and $\matrixC$ is symmetric negative semi-definite.} When the polynomial degree $k\ge 2$, we solve the linear system \eqref{originM} by applying element-wise static condensation, eliminating the local interior degrees of freedom (DOFs) in $\bldV_{h,0}^{k,o}$ and higher order pressure DOFs in $Q_{h,0}^{k-1,o}$, as was done in \cite{LehrenfeldSchoberl16}. Then, the {global} unknowns in the linear system after static condensation {are the DOFs  associated with $\bldV_{h,0}^{k,\partial}$, $\bld{\Vhat}_{h,0}^{k-1}$ and $\overline{Q}_h$}. 

\blue{
{To illustrate the static condensation process}, we denote the coefficient vectors of $(\bldu_h^o,\bldu_h^\partial,\blduhat_h,\overline{p}_{h},$\newline$p_h^o)\in\bldV_{h,0}^{k,o}\times\bldV_{h,0}^{k,\partial}\times\bld{\Vhat}_{h,0}^{k-1}\times\overline{Q}_{h,0}\times Q_{h,0}^{k-1,o}$
in their basis functions as $\matrixU^o,\matrixU^{\partial},\matrixUhat,\overline{\matrixP},\matrixP^o$. {We obtain the following orthogonality relationships through the choice of the finite element spaces, see \eqref{spaceRelation}:}
\begin{alignat*}{3}
b(\overline{p}_{h},\bldu_h^o) &= -(\overline{p}_{h},\nabla\cdot\bldu_h^o)_{\Th}&&= 0,\;\;&&\forall\; \overline{p}_{h}\in\overline{Q}_h, \bldu_h^o\in\bldV_{h}^{k,o},
\\
b(p_h^o,\bldu_h^\partial) &= -(p_h^o,\nabla\cdot\bldu_h^\partial)_{\Th}&&= 0,\;\;&&\forall\;p^o_h\in Q_h^{k-1,o}, \bldu_h^\partial\in \bldV_{h}^{k,\partial},
\\
c(p_h^o, \; \overline{q}_{h}) &=  -\frac{1}{\lambda}(p_h^o,\;\overline{q}_{h})_{\Th} && = 0, \;\;&&\forall\;p^o_h\in Q_h^{k-1,o}, \overline{q}_{h}\in\overline{Q}_h.
\end{alignat*}
With the above facts and by rearranging the order of unknowns in \eqref{originM}, we get:
\begin{equation}
	\label{reorderM}
	\begin{bmatrix}[1.5]
		\matrixA_{u^o u^o}  & \matrixA_{u^o u^\partial} & \matrixA_{u^o\widehat{u}} & \mathbf{0} & \matrixB_{u^o p^o} 
		\\
		\matrixA_{u^\partial u^o}  & \matrixA_{u^\partial u^\partial} & \matrixA_{u^\partial\widehat{u}} & \matrixB_{u^\partial \overline{p}} & \mathbf{0}
		\\
		\matrixA_{\widehat{u} u^o}  & \matrixA_{\widehat{u} u^\partial} & \matrixA_{\widehat{u}\widehat{u}} & \mathbf{0} & \mathbf{0}
		\\
		\mathbf{0}  & \matrixB_{u^\partial \overline{p}}^\trans & \mathbf{0} & \matrixC_{\overline{p}\;\overline{p}} & \mathbf{0}
		\\
		\matrixB_{u^o p^o}^\trans  & \mathbf{0} & \mathbf{0} & \mathbf{0} & \matrixC_{p^o p^o}
	\end{bmatrix}
	\begin{bmatrix}[1.5]
		\matrixU^o \\ 
		\matrixU^{\partial} \\
		\matrixUhat \\
		\overline{\matrixP} \\
		\matrixP^o
	\end{bmatrix}
	=
	\begin{bmatrix}[1.5]
		\mathbf{F}_{u^o} \\
		\mathbf{F}_{u^\partial} \\
		\mathbf{0} \\
		\mathbf{0} \\
		\mathbf{0}
	\end{bmatrix},
\end{equation}

where the subscripts of the matrices represent the subspaces of test and trial functions of the corresponding bilinear forms, the subscripts of $\mathbf{F}$ represent the subspaces of test functions of the linear form. With the following result, we first condense out the local pressure DOFs.}
\begin{lemma}
	\label{lemma:Doo}
	Assume $\matrixU^o,\matrixP^o$ are the solution to the system \eqref{reorderM}, we have:
	\begin{equation*}
		\langle\matrixB_{u^o p^o}\matrixP^o,\matrixV^o\rangle_2=
		\langle\matrixD_{u^o u^o}\matrixU^o,\matrixV^o\rangle_2,\;\;\forall\;\bldv_h^o\in\bldV_{h}^{k,o}.
	\end{equation*}
	where $\langle\matrixD_{u^o u^o}\matrixU^o,\matrixV^o\rangle_2=\lambda(\nabla\cdot\bldu_h^o,\nabla\cdot\bldv_h^o)_{\Th}, \forall\;\bldu_h^o,\bldv_h^o\in\bldV_{h}^{k,o}$.
\end{lemma}
\begin{proof}
	Since we have $\nabla\cdot\bldV_{h}^{k,o}=Q_{h}^{k-1,o}$ in \eqref{spaceRelation}, from the second equation in \eqref{reorderM} we get $\langle\matrixB^\trans_{u^o p^o}\matrixU^o+\matrixC_{p^o p^o}\matrixP^o,\matrixQ^o\rangle_2=(-\nabla\cdot\bldu_h^o-\frac{1}{\lambda}p_h^o,q_h^o)_{\Th} = 0,\;\forall\;q_h^o\in Q_{h}^o$ and we have $p_h^o = -\lambda\nabla\cdot\bldu_h^o$. Then the result follows.
\end{proof}
By Lemma \ref{lemma:Doo}, we eliminate $\mathbf{P}^o$ from \eqref{reorderM} and get:
\begin{equation}
	\label{nopoM}
	\begin{bmatrix}[1.5]
	\matrixA' & \matrixB' \\
	\matrixB'^\trans & \matrixC'	
	\end{bmatrix}
	\begin{bmatrix}[1.5]
		\matrixlineU \\ 
		\overline{\matrixP}
	\end{bmatrix}
	=
	\begin{bmatrix}[1.5]
		\mathbf{F}  \\
		\mathbf{0}
	\end{bmatrix}.
\end{equation}
where
\begin{equation*}
	\matrixA' = 
	\begin{bmatrix}[1.5]
		\matrixA_{u^o u^o}+\mathbf{D}_{u^o u^o} &  \matrixA_{u^o u^\partial} & \matrixA_{u^o \widehat{u}} \\
		\matrixA_{u^\partial u^o} & \matrixA_{u^\partial u^\partial} & \matrixA_{u^\partial \widehat{u}}  \\
		\matrixA_{\widehat{u} u^o}  & \matrixA_{\widehat{u} u^\partial} & \matrixA_{\widehat{u} \widehat{u}} 
	\end{bmatrix}, \quad
	\matrixB' = 
	\begin{bmatrix}[1.5]
		\mathbf{0} \\
		\matrixB_{u^\partial \overline{p}} \\
		 \mathbf{0}
	\end{bmatrix},\quad
	\matrixC'=
	\matrixC_{\overline{p}\;\overline{p}}.
\end{equation*}
The stiffness matrix $\matrixA'$ is still SPD. The corresponding operator formulation of \eqref{nopoM} is to find $(\bldu_h,\overline{p}_{h})$ such that
\begin{align}
	\label{nopoweakForm}
	\left.
	\renewcommand{\arraystretch}{1.5} 
	\begin{tabular}{r l}
		$\lambda(\nabla\cdot\bldu_h^o,\nabla\cdot\bldv_h^o)_{\Th}+a(\lineu_h,\linev_h)+ b(\overline{p}_{h},\bld{v}_h)=$ &
		$\left(\bld{f},\bld{v}_h\right)_{\Th}$ \\
		$b(\bld{u}_h,\overline{q}_h)+c(\overline{p}_{h},\overline{q}_h)=$ & $0$
	\end{tabular}
	\right\},\quad\forall\;\;(\linev_h,\overline{q}_h)\in\lineV_{h,0}^{k}\times \overline{Q}_{h,0},
\end{align}
Similarly, by condensing out $\matrixU^o$ from \eqref{nopoM}{ and denoting the block-diagonal matrix $\mathbf{E} := \matrixA_{u^o u^o}+\mathbf{D}_{u^o u^o}$,} we get the final condensed linear system to solve:
\begin{equation}
	\label{condensedM}
		\begin{bmatrix}[1.5]
		\mathbf{A}_{g} & \mathbf{B}_g \\
		\mathbf{B}^\trans_g & \mathbf{C}_g
	\end{bmatrix}
	\begin{bmatrix}[1.5]
		\matrixlineU_g \\
		\overline{\matrixP}
	\end{bmatrix}
	=
	\begin{bmatrix}[1.5]
		\mathbf{F}_g \\
		\mathbf{0}
	\end{bmatrix},
\end{equation}
where
\begin{align*}
	\matrixA_g &:=
	\begin{bmatrix}[1.5]
		\matrixA_{u^\partial u^\partial}-\matrixA_{u^\partial u^o}\mathbf{E}^{-1}\matrixA_{u^o u^\partial} & \matrixA_{u^\partial \widehat{u}}-\matrixA_{u^\partial u^o}\mathbf{E}^{-1}\matrixA_{u^o \widehat{u}} \\
		\matrixA_{\widehat{u} u^\partial}-\matrixA_{\widehat{u} u^o}\mathbf{E}^{-1}\matrixA_{u^o u^\partial} & \matrixA_{\widehat{u}\widehat{u}}-\matrixA_{\widehat{u} u^o}\mathbf{E}^{-1}\matrixA_{u^o \widehat{u}}
	\end{bmatrix},\quad
	\matrixB_g :=
	\begin{bmatrix}[1.5]
		\matrixB_{u^\partial \overline{p}} \\
		 \mathbf{0}
	\end{bmatrix}, \quad
	\matrixC_g := \matrixC_{\overline{p}\;\overline{p}},
	\\
	\matrixlineU_g &:= 
	\begin{bmatrix}
		\matrixU^\partial  & \matrixUhat
	\end{bmatrix}^\trans,\quad
	\mathbf{F}_g :=
	\begin{bmatrix}
		\mathbf{F}_{u^\partial}-\matrixA_{u^\partial u^o}\mathbf{E}^{-1}\mathbf{F}_{u^o}
		&
		-\matrixA_{\widehat{u} u^o}\mathbf{E}^{-1}\mathbf{F}_{u^o}
	\end{bmatrix}^\trans.
\end{align*}
\revise{Since basis functions of $\bld{V}_h^{k,o}$ are locally supported, the matrix $\mathbf{E}$ is block-diagonal, and the inverse can be obtained element-wise.}

\section{block-diagonal Preconditioners for the Condensed System}
\label{sec:precondition}
In this section, we construct robust and optimal preconditioners for the condensed $H$(div)-conforming HDG scheme in \eqref{condensedM}. From \cite{pestana2015natural, benzi2008some} an ideal block-diagonal preconditioner for the saddle point system \eqref{condensedM} is
\begin{equation*}
	\begin{bmatrix}[1.5]
		\mathbf{A}_g^{-1} &  \\
		 & \mathbf{S}_g^{-1}
	\end{bmatrix},
\end{equation*}
where $\mathbf{S}_g$ is the following (negative) Schur complement
\begin{equation*}
	\mathbf{S}_g  = -\mathbf{C}_g+\mathbf{B}_g^\trans\mathbf{A}_g^{-1}\mathbf{B}_g.
\end{equation*}
Obviously, it is not practical to compute the dense matrix {$\mathbf{S}_g$} directly. {In practice, we seek computable approximations to the two matrix inverses $\matrixA_g^{-1}$ and $\mathbf{S}_g^{-1}$.} For the stiffness matrix $\mathbf{A}_g$, we use an ASP with continuous element-wise linear function space as the auxiliary space. For the Schur complement $\mathbf{S}_g$, we get the explicit expression of a spectrally equivalent matrix $\widetilde{\mathbf{S}}_g$ inspired by \cite{Mardal04, Olshanskii06}.

\subsection{Preconditioner for the stiffness matrix $\matrixA_g$}
We extend from the work in our previous study \cite{fu2021uniform} on the ASP for the divergence-conforming HDG scheme for the reaction-diffusion equations and apply it to the stiffness matrix $\matrixA_g$ here.
Since the analysis procedure is almost the same as \cite[Section 3.3]{fu2021uniform}, we quote from it and sketch our main steps here.
\begin{enumerate}[label=(\roman*),leftmargin= 15 pt]
	\item
	We starts from \eqref{nopoweakForm} to get the operator formulation of $\matrixA_g$. We define the following mapping $\pazocal{L}^o_h:\bldV_{h,0}^{k,\partial}\times\bldVhat_{h,0}^{k-1}\rightarrow\bldV_{h,0}^{k,o}$: Given $(\bldu_h^\partial,\blduhat_h)\in\bldV_{h,0}^{k,\partial}\times\bldVhat_{h,0}^{k-1}$, $\pazocal{L}^o_h(\bldu_h^\partial,\blduhat_h)\in\bldV_{h,0}^{k,o}$ is the unique solution that satisfies
	\begin{align}
	\label{eqn:liftOpt}
		\lambda\left(\nabla\cdot\pazocal{L}^o_h(\bldu_h^\partial,\blduhat_h),\nabla\cdot\bldv_h^o\right)_{\Th}&+a\left((\pazocal{L}^o_h(\bldu_h^\partial,\blduhat_h),0),(\bldv_h^o,0)\right)
		\\
		\nonumber
		&=-a\left((\bldu_h^\partial,\blduhat_h),(\bldv_h^o,0)\right),
	\end{align}
	for all $\bldv_h^o\in\bldV_{h,0}^{k,o}$. Then, we express $\matrixA_g$ in the final condensed HDG scheme as
	\begin{align*}
		\langle\matrixA_g\underline{\matrixU}_g,\underline{\matrixV}_g\rangle_2
		=&
		\lambda\left(\nabla\cdot\pazocal{L}^o_h(\bldu_h^\partial,\blduhat_h),\nabla\cdot\pazocal{L}^o_h(\bldv_h^\partial,\bldvhat_h)\right)_{\Th} \\
		&+a\left((\bldu_h^\partial+\pazocal{L}^o_h(\bldu_h^\partial,\blduhat_h),\blduhat_h),(\bldv_h^\partial+\pazocal{L}^o_h(\bldv_h^\partial,\bldvhat_h),\bldvhat_h)\right),
	\end{align*}
	for all $\bldu_h^\partial,\bldv_h^\partial\in\bldV_{h,0}^{k,\partial}$, $\blduhat_h, \bldvhat_h\in\bldVhat_{h,0}^{k-1}$. \revise{Matrix $\matrixA_g$ is SPD by the definition of the corresponding bilinear form, and we denote the norm defined by it as}
	\begin{align*}
		\|(\bldu_h^\partial,\blduhat_h)\|_{\matrixA_g}^2
		:=&
		\langle\matrixA_g\underline{\matrixU}_g,\underline{\matrixU}_g\rangle_2.
	\end{align*}
	{
	Taking $\bld{v}_h^o = \pazocal{L}^o_h(\bldu_h^\partial, \blduhat_h)$ and applying the Cauchy Schwarz inequality in \eqref{eqn:liftOpt}, we have:
	\begin{align*}
	    \lambda\left(\nabla\cdot\pazocal{L}^o_h(\bldu_h^\partial,\blduhat_h),\nabla\cdot\pazocal{L}^o_h(\bldu_h^\partial,\blduhat_h)\right)_{\Th}
	    &+
	    a\left((\pazocal{L}^o_h(\bldu_h^\partial,\blduhat_h),0),(\pazocal{L}^o_h(\bldu_h^\partial,\blduhat_h),0)\right)
	    \\
	    &\leq
	    a\left((\bldu_h^\partial,\blduhat_h),(\bldu_h^\partial,\blduhat_h)\right).
	\end{align*}
	Hence, by the triangle inequality we get:
	\begin{align*}
	    \|(\bldu_h^\partial,\blduhat_h)\|_{\matrixA_g}^2
	    \leq&
	    \lambda\left(\nabla\cdot\pazocal{L}^o_h(\bldu_h^\partial,\blduhat_h),\nabla\cdot\pazocal{L}^o_h(\bldu_h^\partial,\blduhat_h)\right)_{\Th}
	    +
	    a\left((\pazocal{L}^o_h(\bldu_h^\partial,\blduhat_h),0),(\pazocal{L}^o_h(\bldu_h^\partial,\blduhat_h),0)\right)
	    \\
	    &+
	    a\left((\bldu_h^\partial,\blduhat_h),(\bldu_h^\partial,\blduhat_h)\right)
	    \\
	    \leq&
	    a\left((\bldu_h^\partial,\blduhat_h),(\bldu_h^\partial,\blduhat_h)\right)
	    \\
	    \simeq&
	    |\!|\!|(\bldu_h^\partial, \blduhat_h)|\!|\!|_{\ast,h}^2.
	\end{align*}
	}
	
	\item
	Next, we define an $L^2$-like inner product on the compound space $\bldV_{h,0}^{k,\partial}\times\bldVhat_{h,0}^{k-1}$:
	\begin{align*}
		\left((\bldu_h^\partial,\blduhat_h),(\bldv_h^\partial,\bldvhat_h)\right)_{0,h}
		:=
		(2\mu+\tau h^2)\left(\bldu_h^\partial,\bldv_h^\partial\right)_{\Th}+2\mu h\left\langle\mathsf{tang}(\blduhat_h),\mathsf{tang}(\bldvhat_h)\right\rangle_{\Eh},
	\end{align*}
	and denote its corresponding norm as $\|(\bldu_h^\partial,\blduhat_h)\|_{0,h}^2:=\left((\bldu_h^\partial,\blduhat_h),(\bldu_h^\partial,\blduhat_h)\right)_{0,h}$. Then, we have the following result:
	\begin{align*}
		\rho_{\matrixA_g}&\simeq h^{-2},
	\end{align*}
	where $\rho_{\matrixA_g}=\rho(\matrixA_g)$ denotes the spectral radius of $\matrixA_g$.
	
	\item
	Denote the diagonal matrix $\mathbf{D}_g$ with the same diagonal components of $\matrixA_g$. For the linear operator $j_g:\bldV_{h,0}^{k,\partial}\times\bldVhat_{h,0}^{k-1}\rightarrow\bldV_{h,0}^{k,\partial}\times\bldVhat_{h,0}^{k-1}$ corresponding to the Jacobi smoother $\mathbf{R}_g=\mathbf{D}_g^{-1}$, we have:
	\begin{equation*}
		\left(j_g(\bldu_h^\partial,\blduhat_h),(\bldu_h^\partial,\blduhat_h)\right)_{0,h}
		\simeq
		\rho_{\matrixA_g}^{-1}\|(\bldu_h^\partial,\blduhat_h)\|_{0,h}^2.
	\end{equation*}
	
	\item
	We define a continuous element-wise linear finite element space: 
	\[
		\pazocal{V}_{h,0}^1:=\{\bldv_0\in[H^1(\Omega)]^d:\bldv_0|_K\in\pazocal{P}^1(K),\forall K\in\Th,\bldv_0|_F=0,\forall F\in\Eh^\partial\},
	\]
	and use it as the auxiliary space. The matrix $\matrixA_0$ and bilinear operator $a_0$ on $\pazocal{V}_{h,0}^1$ corresponding to $\matrixA_g$ is defined as
	\begin{align*}
		\langle\matrixA_0\matrixU_0,\matrixV_0\rangle_2 = a_0(\bldu_0,\bldv_0):= \int_\Omega \left(2\mu\nabla(\bldu_0):\nabla(\bldv_0)+\tau\bldu_0\cdot\bldv_0\right)\mathsf{dx}.
	\end{align*}
	\revise{$\matrixA_0$ is SPD by the definition of the corresponding bilinear form, and we define the induced norm on $\pazocal{V}_{h,0}^1$ by $\|\cdot\|_{\matrixA_0}^2:=a_0(\cdot,\cdot)$.} 
	$\matrixA_0$ can be easily preconditioned by an algebraic or geometric multigrid procedure.
	
	\item
	We define the operator $\underline{{\bld{\Pi}}}_h=(\bld{\Pi}_h^\partial,\widehat{\bld{\Pi}}_h):\pazocal{V}_{h,0}^1\rightarrow\bldV_{h,0}^{k,\partial}\times\bldVhat_{h,0}^{k-1}$ by
	\begin{align*}
		\langle\bld{\Pi}_h^\partial\bldu_0\cdot\bld{n},\bldv_h^\partial\cdot\bld{n}\rangle_{\Eh}
		&=
		\langle\bldu_0\cdot\bld{n},\bldv_h^\partial\cdot\bld{n}\rangle_{\Eh},\quad\forall\bldv_h^\partial\in\bldV_{h,0}^{k,\partial}, \\
		\langle\mathsf{tang}(\widehat{\bld{\Pi}}_h\bldu_0),\mathsf{tang}(\bldvhat_h)\rangle_{\Eh}
		&=
		\langle\mathsf{tang}(\bldu_0),\mathsf{tang}(\bldvhat_h)\rangle_{\Eh},\quad\forall\bldvhat_h\in\bldVhat_{h,0}^{k-1}.
	\end{align*}
	The operator $\bld{P}_h:\bldV_{h,0}^{k,\partial}\times\bldVhat_{h,0}^{k-1}\rightarrow\pazocal{V}_{h,0}^1$ is defined on mesh vertices by
	\begin{align*}
		\bld{P}_h(\bldu_h^\partial,\blduhat_h)(\bld{x}_n)=\left\{
		\begin{tabular}{l l}
			$0$, & if $\bld{x}_n\in\partial\Omega$,  \\
			$\frac{1}{\# K_n}\sum_{K\in K_n}\left(\bldu_h^\partial+\pazocal{L}^o_h(\bldu_h^\partial,\blduhat_h)\right)|_K(\bld{x}_n)$, &
			if $\bld{x}_n\notin\partial\Omega$,
		\end{tabular}
		\right.
	\end{align*}
	where $\bld{x_n}$ is a vertex of $\Th$, $K_n$ is the set of elements of $\Th$ that share the vertex $\bld{x_n}$ and $\# K_n$ is the cardinality of it. {We note that the operator $\bld{P}_h$ is only used for analysis and does not appear in the computing process}. Then, we have the following boundedness properties:
	\begin{subequations}
	    \label{eqn:ASP_condition}
    	\begin{align}
    		\|\underline{\bld{\Pi}}_h\bldu_0\|_{\matrixA_g}&\lesssim
    		\|\bldu_0\|_{\matrixA_0},
    		\\
    		\|\bld{P}_h(\bldu_h^\partial,\blduhat_h)\|_{\matrixA_0}&\lesssim
    		\|(\bldu_h^\partial,\blduhat_h)\|_{\matrixA_g},
    		\\
    		\|(\bldu_h^\partial,\blduhat_h)-\underline{\bld{\Pi}}_h\bld{P}_h(\bldu_h^\partial,\blduhat_h)\|_{0,h}&\lesssim
    		\rho_{\matrixA_g}^{-1/2}\|(\bldu_h^\partial,\blduhat_h)\|_{\matrixA_g}.
    	\end{align}
	\end{subequations}
\end{enumerate}
Finally, we denote $\kappa$ as the condition number and obtain the optimality of our auxiliary space preconditioner in the following theorem by invoking \cite[Theorem 2.1]{Xu96} combined with the above results.

\begin{theorem}[ASP for the stiffness matrix]
	\label{theo:ASPpre}
	Let 
	\begin{align}
	\label{asp}
	\widetilde{\matrixA}_g^{-1}=\mathbf{R}_g+\underline{\bld{\Pi}}_h\bld{B_0}\underline{\bld{\Pi}}_h^\trans	    
	\end{align}
	be the auxiliary space preconditioner for the operator $\matrixA_g$ in the final condensed HDG scheme \eqref{condensedM} with $\kappa(\matrixB_0\matrixA_0)\simeq 1$, then we have $\kappa(\widetilde{\matrixA}_g^{-1}\matrixA_g)\simeq 1$.
\end{theorem}

{
\begin{remark}
    We note that there exist two differences between the ASP applied here and the one in our previous work \cite{fu2021uniform}. 
    Firstly, there is an extra term $\lambda(\nabla\cdot\bldu_h^o,\nabla\cdot\bldv_h^o)_{\Th}$ in the bilinear operator corresponding to $\matrixA_g$ in this study, which makes the norm $\|\cdot\|_{\mathbf{A}_g}$ stronger but still bounded by $|\!|\!|(\bldu_h^\partial, \blduhat_h)|\!|\!|_{\ast,h}$ as mentioned in step (i).
    Due to the auxiliary space $\pazocal{V}_{h,0}^1\subset \bldV_{h,0}^{k,\partial}$, this term does not appear in the bilinear operator $a_0(\cdot,\cdot)$.
    Secondly, symmetric gradient is used in the bilinear operator $a(\cdot,\cdot)$ instead of gradient. The discrete Korn's inequality holds as in \cite{brenner2004korn} and we have the norm equivalence $\sum_{K\in\Th}\|\bld{D}(\bldu_h)\|_K^2+\frac{1}{h}\|\mathsf{tang}(\bldu_h-\blduhat_h)\|_{\partial K}^2
    \simeq
    \sum_{K\in\Th}\|\nabla(\bldu_h)\|_K^2+\frac{1}{h}\|\mathsf{tang}(\bldu_h-\blduhat_h)\|_{\partial K}^2$. Therefore, this difference does not affect the results of our analysis as well.
\end{remark}
}

\subsection{Preconditioner for the Schur complement $\mathbf{S}_g$}
\blue{
Since we use the concept of sum and intersection of Hilbert spaces in this subsection, we briefly introduce the definition and basic properties here, {see \cite[Section 2.2]{Olshanskii06} and \cite[Chapter 2]{Bergh:1617905}}. Assume $X$ and $Y$ are compatible Hilbert spaces both continuously contained in some larger Hilbert space, then their intersection $X\cap Y$ and their sum $X+Y$ are both complete Hilbert spaces with corresponding norms:
\begin{alignat}{3}
	\label{interNorm}
	&\|z\|_{X\cap Y}^2 &&=\;
	\|z\|_X^2 + \|z\|_Y^2,\;\;&&\forall\; z\in X\cap Y,
	\\
	\label{sumNorm}
	&\|z\|_{X+Y}^2 &&=\;
	\inf(\|x\|_X^2+\|y\|_Y^2),\;\;&&\forall\; z=x+y,x\in X, y\in Y.
\end{alignat}
If $X_1$ and $Y_1$ as normed vector spaces, we denote $X_1'$ as the dual space of $X_1$, $\pazocal{L}(X_1,Y_1)$ as the space of bounded linear mapping from $X_1$ to $Y_1$. Then, we have the following properties, the proof of which we refer to \cite[Chapter 2]{Bergh:1617905}.
\begin{lemma}
	\label{lemma:mapbound}
	Assume $X_1,X_2$ and $Y_1,Y_2$ are pairs of compatible normed vector spaces. If the linear mapping $T\in \pazocal{L}(X_1,Y_1)\cap \pazocal{L}(X_2,Y_2)$, then we have:
\begin{align}
	\label{mapBound}
	\|T\|_{X_1+X_2\rightarrow Y_1+Y_2}\le
	(\|T\|_{X_1\rightarrow Y_1}^2+\|T\|_{X_2\rightarrow Y_2}^2)^{1/2}.
\end{align}
If both $X$ and $Y$ are Hilbert spaces, given any $g\in(X+Y)'$, we have:
\begin{align}
	\label{dualTrans}
	\|g\|_{(X+Y)'}=
	\|g\|_{X'+Y'}.
\end{align}
\end{lemma}

The Schur complement of the final condensed system \eqref{condensedM} appears to be too complicated to be analyzed. However, in the static condensation process, the Schur complement stays the same before and after $\matrixU^o$ is condensed out, which significantly simplifies our analysis. A similar idea was used in \cite{rhebergen2018preconditioning} for a different HDG scheme for the Stokes problem.
We first quote a lemma about block matrix inverse, the proof of which we refer to \cite[Chapter 0.7.3]{johnson1985matrix}:
\begin{lemma}
	\label{lemma:matInv}
	Assume a block matrix $\mathbf{R}$ in the form of
	\begin{equation*}
		\mathbf{R}=
		\begin{bmatrix}[1.5]
			\matrixA & \matrixB \\
			\matrixC & \matrixD
		\end{bmatrix},
	\end{equation*}
	and the submatrices $\matrixA$ and $\matrixD$ are both invertible, then the inverse of $\mathbf{R}$ is expressed as
	\begin{equation*}
		\mathbf{R}^{-1}=
		\begin{bmatrix}[1.5]
			(\matrixA-\matrixB\matrixD^{-1}\matrixC)^{-1} & -\matrixA^{-1}\matrixB(\matrixD-\matrixC\matrixA^{-1}\matrixB)^{-1} \\
			-(\matrixD-\matrixC\matrixA^{-1}\matrixB)^{-1}\matrixC\matrixA^{-1} & (\matrixD-\matrixC\matrixA^{-1}\matrixB)^{-1}
		\end{bmatrix}.
	\end{equation*}
\end{lemma}
Then, we prove the following result:
\begin{lemma}
	\label{lem:equivSchur}
	The Schur complement of the HDG scheme with only $\matrixP^o$ condensed out in \eqref{nopoM} is the same as that of the final condensed system \eqref{condensedM}.
\end{lemma}
\begin{proof}
		Since the stiffness matrix $\matrixA'$ is SPD, it is straightforward to verify that the matrix remains SPD {after $\matrixU^o$ is condensed out.}
		 By Lemma \ref{lemma:matInv}, the (negative) Schur complement on the space $\overline{Q}_h$ of the HDG scheme in $\eqref{nopoM}$ is expressed as
		\begin{align*}
			\mathbf{S}' &= -\matrixC_{\overline{p}\;\overline{p}}+
			\begin{bmatrix}[1.5]
				\mathbf{0} & \matrixB_{u^\partial \overline{p}}^\trans & \mathbf{0}
			\end{bmatrix}
			\begin{bmatrix}[1.5]
				\matrixA_{u^o u^o}+\mathbf{D}_{u^o u^o} &  \matrixA_{u^o u^\partial} & \matrixA_{u^o \widehat{u}} \\
				\matrixA_{u^\partial u^o} & \matrixA_{u^\partial u^\partial} & \matrixA_{u^\partial \widehat{u}} \\
				\matrixA_{\widehat{u} u^o}  & \matrixA_{\widehat{u} u^\partial} & \matrixA_{\widehat{u}\widehat{u}}
			\end{bmatrix}^{-1}
			\begin{bmatrix}[1.5]
				\mathbf{0} \\
				\matrixB_{u^\partial \overline{p}} \\
				\mathbf{0}
			\end{bmatrix}\\
			&=
			-\matrixC_g + \matrixB_g^\trans
			(
			\begin{bmatrix}[1.5]
				\matrixA_{u^\partial u^\partial} & \matrixA_{u^\partial \widehat{u}} \\
				\matrixA_{\widehat{u} u^\partial} & \matrixA_{\widehat{u}\widehat{u}}
			\end{bmatrix}
			-
			\begin{bmatrix}[1.5]
				\matrixA_{u^\partial u^o} \\
				\matrixA_{\widehat{u} u^o}
			\end{bmatrix}
			\mathbf{E}^{-1}
			\begin{bmatrix}
				\matrixA_{u^o u^\partial} & \matrixA_{u^o \widehat{u}}
			\end{bmatrix}
			)^{-1}
			\matrixB_g\\
			&=
			-\matrixC_g+\matrixB_g^\trans\matrixA_g\matrixB_g,
		\end{align*}
		{where $\mathbf{E}=\matrixA_{u^o u^o}+\mathbf{D}_{u^o u^o}$}.
\end{proof}

From now on, we directly work with the (negative) Schur complement $\mathbf{S'}$ of the HDG scheme in \eqref{nopoM}, which is the same as $\mathbf{S}_g$ from Lemma \ref{lem:equivSchur}. Inspired by the ideas in \cite{Mardal04, Olshanskii06}, we need to define a parameter-dependent norm on the finite element space $\overline{Q}_h$ to obtain the upper and lower spectral bound of $\mathbf{S}'$ independent of model parameters and mesh size.
We start from the norm defined by the SPD (negative) Schur complement $\mathbf{S}'$.
\begin{lemma}
    \label{lem:Snorm}
    For all $\overline{p}_h\in\overline{Q}_h$, we have:
    \begin{align*}
		\langle\mathbf{S}'\overline{\matrixP},\overline{\matrixP}\rangle_2
		&\simeq
		\frac{1}{\lambda}\|\overline{p}_{h}\|_{\Th}^2 + \sup_{\linev_h\in\lineV_{h,0}^k}\frac{(\overline{p}_{h},\nabla\cdot\bldv_h)^2_{\Th}}{\tau\|\bldv_h\|_{\Th}^2+2\mu|\!|\!|\underline{\bldv}_h|\!|\!|_{1,h}^2+\lambda\|\nabla\cdot\bldv_h^o\|_{\Th}^2}.
	\end{align*}
\end{lemma}
\begin{proof}
    Since $\matrixA'$ in \eqref{nopoM} is SPD, $\matrixA_{}^{\prime -1/2}$ is also SPD. Denote $n$ as the number of DOFs of $\underline{\bldV}_{h,0}^k$. Combined with the definition of matrix representations of bilinear forms and the norm defined by $a(\cdot,\;\cdot)$ in \eqref{aNorm}, we have:
	\begin{align*}
		\langle\mathbf{S}'\overline{\matrixP},\overline{\matrixP}\rangle_2
		&=
		\langle(-\matrixC'+\matrixB'^\trans\matrixA^{\prime -1}\matrixB')\overline{\matrixP},\overline{\matrixP}\rangle_2\\
		&=
		\frac{1}{\lambda}\|\overline{p}_{h}\|_{\Th}^2+\langle\matrixA^{\prime -1/2}\matrixB'\overline{\matrixP},\matrixA^{\prime -1/2}\matrixB'\overline{\matrixP}\rangle_2 \\
		&=
		\frac{1}{\lambda}\|\overline{p}_{h}\|_{\Th}^2 + \sup_{\matrixlineV\in\mathbb{R}^{n}}\frac{\langle\matrixA^{\prime -1/2}\matrixB'\overline{\matrixP},\matrixlineV\rangle_2^2}{\langle\matrixlineV,\matrixlineV\rangle_2} \\
		&=
		\frac{1}{\lambda}\|\overline{p}_{h}\|_{\Th}^2 + \sup_{\matrixlineV\in\mathbb{R}^{n}}\frac{\langle\matrixB'\overline{\matrixP},\matrixA^{\prime -1/2}\matrixlineV\rangle_2^2}{\langle\matrixlineV,\matrixlineV\rangle_2} \\
		&=
		\frac{1}{\lambda}\|\overline{p}_{h}\|_{\Th}^2 + \sup_{\underline{\mathbf{W}}\in\mathbb{R}^{n}}\frac{\langle\matrixB'\overline{\matrixP},\underline{\mathbf{W}}\rangle_2^2}{\langle\matrixA^{\prime}\underline{\mathbf{W}},\underline{\mathbf{W}}\rangle_2} \\
		&\simeq
		\frac{1}{\lambda}\|\overline{p}_{h}\|_{\Th}^2 + \sup_{\linev_h\in\lineV_{h,0}^k}\frac{(\overline{p}_{h},\nabla\cdot\bldv_h)^2_{\Th}}{\tau\|\bldv_h\|_{\Th}^2+2\mu|\!|\!|\underline{\bldv}_h|\!|\!|_{1,h}^2+\lambda\|\nabla\cdot\bldv_h^o\|_{\Th}^2}.
	\end{align*}
\end{proof}

Next, we introduce a parameter-dependent norm on the element-wise constant space $\overline{Q}_h$:
\begin{equation}
	\label{pbarNorm}
	|\overline{p}_{h}|_\ast^2:=
	\frac{1}{\lambda}\|\overline{p}_{h}\|_{\Th}^2+\inf_{\overline{p}_{h}=\overline{p}_{h,1}+\overline{p}_{h,2}}\left(\frac{1}{2\mu}\|\overline{p}_{h,1}\|_{\Th}^2+\frac{1}{\tau h}\|\jmp{\overline{p}_{h,2}}\|_{\Eh}^2\right).
\end{equation}
Before proceeding to the parameter-independent stability and boundedness of $\mathbf{S}'$ with respect to $|\overline{p}_{h}|_\ast$, we need to prove two inf-sup conditions.

\begin{lemma}
    \label{lem:infsup1}
    For all $\overline{p}_h\in\overline{Q}_h$, we have:
    \begin{align*}
        \sup_{\linev_h\in\lineV_{h,0}^k}\frac{(\overline{p}_{h},\nabla\cdot\bldv_h)_{\Th}^2}{2\mu|\!|\!|\underline{\bldv}_h|\!|\!|_{1,h}^2+\lambda\|\nabla\cdot\bldv_h^o\|_{\Th}^2}
        \gtrsim
        \frac{1}{2\mu}\|\overline{p}_{h}\|_{\Th}^2.
    \end{align*}
\end{lemma}
\begin{proof}
    The proof procedure is similar to \cite[Proposition 2.3.5]{Lehrenfeld10}. 
    {For any $\overline{p}_h\in\overline{Q}_h$, assume $\phi\in H^2(\Omega)$ satisfies the Neumann problem $-\Delta\phi=\overline{p}_h$ with $\frac{\partial\phi}{\partial n}=0$ on $\partial\Omega$. We take $\bld{w}^\ast\in[H^1(\Omega)]^d$ such that $\bld{w}^\ast=\nabla\phi$ and we get $\|\bld{w}^\ast\|_{H^1}\lesssim\|\phi\|_{H^2}\lesssim\|\overline{p}_h\|_{L^2}$ from the elliptic regularity due to the convex domain $\Omega$. Then from \cite[Lemma 2.3.1]{Lehrenfeld10} there exists $\bld{w}_h\in\bldV_{h,0}^k$ satisfying $(q_h, \nabla\cdot\bld{w}_h)_{\Th}=(q_h, \nabla\cdot\bld{w}^\ast)_{\Th}$ for all $q_h\in Q_h^{k-1}$ and 
    $|\!|\!|\bld{w}_h|\!|\!|_{1,h}\lesssim\|\bld{w}^\ast\|^2_{H^1}$}.
    Since $k\ge 1$ in the $H$(div)-conforming HDG scheme and $\bldv_h^o=\bld{0}$ in the lowest order case, we have:
    \begin{align*}
    	\sup_{\linev_h\in\lineV_{h,0}^k}\frac{(\overline{p}_{h},\nabla\cdot\bldv_h)_{\Th}^2}{2\mu|\!|\!|\underline{\bldv}_h|\!|\!|_{1,h}^2+\lambda\|\nabla\cdot\bldv_h^o\|_{\Th}^2}
    	&\ge
    	\sup_{\linev_h\in\lineV_{h,0}^1}\frac{(\overline{p}_{h},\nabla\cdot\bldv_h)_{\Th}^2}{2\mu|\!|\!|\underline{\bldv}_h|\!|\!|_{1,h}^2} \\
    	&\ge
    	\frac{(\overline{p}_{h},\nabla\cdot\bld{w}_h)_{\Th}^2}{2\mu|\!|\!|\underline{\bld{w}}_h|\!|\!|_{1,h}^2}
    	\\
    	&\gtrsim
    	\frac{(\overline{p}_{h},\nabla\cdot\bld{w}^\ast)_{\Th}^2}{2\mu\|\bld{w}^\ast\|_{H^1}^2}
    	\\
    	&\gtrsim
    	\frac{\|\overline{p}_{h}\|_{\Th}^4}{2\mu\|\overline{p}_{h}\|_{\Th}^2} \\
    	&=
    	\frac{1}{2\mu}\|\overline{p}_{h}\|_{\Th}^2.
    \end{align*}
\end{proof}

\begin{lemma}
    \label{lem:infsup2}
    For all $\overline{p}_h\in\overline{Q}_h$, we have:
    \begin{align*}
        \sup_{\bldv_h\in\bldV_{h,0}^k}\frac{(\overline{p}_h,\nabla\cdot\bldv_h)_{\Th}^2}{\tau\|\bldv_h\|_{\Th}^2}
        &\gtrsim
        \frac{1}{\tau h}\|\jmp{\overline{p}_h}\|_{\Eh}^2.
    \end{align*}
\end{lemma}
\begin{proof}
    For any $\overline{p}_h\in\overline{Q}_h$, we construct $\bld{w}_h\in\mathrm{RT0}\subset\bldV_{h}^k$ such that $\bld{w}_h\cdot\bld{n}|_F=\jmp{\overline{p}_h}|_F$ for all $F\in\Eh^o$.
	By norm equivalence and standard scaling arguments, we have $\|\bld{w}_h\|_{\Th}^2\simeq h\|\bld{w}_h\cdot\bld{n}\|_{\Eh}^2$. Since normal components of functions in $\bldV_{h,0}^k$ are continuous across element facets, by integrating by parts we have:
    \begin{align*}
	\sup_{\bldv_h\in\bldV_{h,0}^k}\frac{(\overline{p}_h,\nabla\cdot\bldv_h)_{\Th}^2}{\tau\|\bldv_h\|_{\Th}^2}
	&=
	\sup_{\bldv_h\in\bldV_{h,0}^k}\frac{\langle\jmp{\overline{p}_h},\bldv_h\cdot\bld{n}\rangle_{\Eh}^2}{\tau\|\bldv_h\|_{\Th}^2} \\
	&\ge
	\frac{\langle\jmp{\overline{p}_h},\bld{w}_h\cdot\bld{n}\rangle_{\Eh}^2}{\tau\|\bld{w}_h\|_{\Th}^2} \\
	&\ge
	\frac{\|\jmp{\overline{p}_h}\|_{\Eh}^4}{\tau\|\bld{w}_h\|_{\Th}^2} \\
	&{\simeq}
	\frac{\|\overline{p}_h\|_{\Eh}^4}{\tau h\|\bld{w}_h\cdot\bld{n}\|_{\Eh}^2} \\
	&=
	\frac{1}{\tau h}\|\jmp{\overline{p}_h}\|_{\Eh}^2.
\end{align*}
\end{proof}

We are now ready to present the equivalence between the newly defined norm and the one induced by $\mathbf{S}'$ with the above properties.
\begin{theorem}[Equivalent Schur complement norm]
	\label{theo:SchurEquivNorm}
	\[
		\langle\mathbf{S}'\overline{\matrixP},\overline{\matrixP}\rangle_2
		\simeq
		|\overline{p}_{h}|_\ast^2,\quad\forall \overline{p}_{h}\in\overline{Q}_h.
	\]
\end{theorem}
\begin{proof}
\begin{enumerate}[label=(\alph*),leftmargin = 15 pt]
\item
To prove $\langle\mathbf{S}'\overline{\matrixP},\overline{\matrixP}\rangle_2
\gtrsim
|\overline{p}_{h}|_\ast^2$, we denote two normed vector spaces:
\begin{alignat*}{2}
	{X}_1&=\lineV_{h,0}^k,\;\;\|\bldv_h\|_{{X}_1}^2 = 2\mu|\!|\!|\underline{\bldv}_h|\!|\!|_{1,h}^2+\lambda\|\nabla\cdot\bldv_h^o\|_{\Th}^2,\quad&&\forall\bldv_h\in {X}_1, \\
	{X}_2&=\bldV_{h,0}^k,\;\;\|\bldv_h\|_{{X}_2}^2 = \tau\|\bldv_h\|_{\Th}^2,\quad&&\forall\bldv_h\in{X}_2,
\end{alignat*}
and the mapping $T\in\pazocal{L}(\overline{Q}_h,X_1')\cap\pazocal{L}(\overline{Q}_h,X_2')$ such that\[(T\overline{p}_{h},\bldv_h)=(\overline{p}_{h},\nabla\cdot\bldv_h),\quad\forall\overline{p}_{h}\in\overline{Q}_h,\bldv_h\in \bldV_{h,0}^k.\] 
With the norm defined by $\mathbf{S}'$ in Lemma \ref{lem:Snorm}, the inf-sup conditions in Lemma \ref{lem:infsup1} and Lemma \ref{lem:infsup2}, the definition of sum and intersection Hilbert spaces in \eqref{interNorm} and \eqref{sumNorm}, linear mapping properties in Lemma \ref{lemma:mapbound}, we have the following result:
\begin{align*}
	\langle\mathbf{S}'\overline{\matrixP},\overline{\matrixP}\rangle_2
	&\simeq
	\frac{1}{\lambda}\|\overline{p}_{h}\|_{\Th}^2+
	\sup_{\linev_h\in\lineV_{h,0}^k}\frac{(\overline{p}_{h},\nabla\cdot\bldv_h)^2_{\Th}}{\tau\|\bldv_h\|_{\Th}^2+2\mu|\!|\!|\underline{\bldv}_h|\!|\!|_{1,h}^2+\lambda\|\nabla\cdot\bldv_h^o\|_{\Th}^2}
	\\
	&=
	\frac{1}{\lambda}\|\overline{p}_{h}\|_{\Th}^2+
	\|T\overline{p}_h\|_{(X_1\cap X_2)'} \\
	&=
	\frac{1}{\lambda}\|\overline{p}_{h}\|_{\Th}^2+
	\|T\overline{p}_h\|_{X_1'+X_2'} \\
	&\gtrsim
	|\overline{p}_h|_{\ast}^2.
\end{align*}

\item
To prove $\langle\mathbf{S}\overline{\matrixP},\overline{\matrixP}\rangle_2
\lesssim
|\overline{p}_{h}|_\ast^2$, 
we assume an arbitrary splitting $\overline{p}_{h}=\overline{p}_{h,1}+\overline{p}_{h,2}$. By using the norm defined by $\mathbf{S}'$ in Lemma \ref{lem:Snorm}, integration by parts, Cauchy Schwarz inequality, and inverse inequality, we have:
\begin{align*}
		\langle\mathbf{S}'\overline{\matrixP},\overline{\matrixP}\rangle_2
		&\simeq
		\frac{1}{\lambda}\|\overline{p}_{h}\|_{\Th}^2 + \sup_{\linev_h\in\lineV_{h,0}^k}\frac{(\overline{p}_{h},\nabla\cdot\bldv_h)^2_{\Th}}{\tau\|\bldv_h\|_{\Th}^2+2\mu|\!|\!|\underline{\bldv}_h|\!|\!|_{1,h}^2+\lambda\|\nabla\cdot\bldv_h^o\|_{\Th}^2} \\
		&\lesssim
		\frac{1}{\lambda}\|\overline{p}_{h}\|_{\Th}^2 + \sup_{\linev_h\in\lineV_{h,0}^k}\frac{(\overline{p}_{h,1},\nabla\cdot\bldv_h)^2_{\Th}+(\overline{p}_{h,2},\nabla\cdot\bldv_h)^2_{\Th}}{\tau\|\bldv_h\|_{\Th}^2+2\mu|\!|\!|\underline{\bldv}_h|\!|\!|_{1,h}^2} \\
		&=
		\frac{1}{\lambda}\|\overline{p}_{h}\|_{\Th}^2 + \sup_{\linev_h\in\lineV_{h,0}^k}\frac{(\overline{p}_{h,1},\nabla\cdot\bldv_h)^2_{\Th}+\langle\jmp{\overline{p}_{h,2}},\bldv_h\bld{n}\rangle^2_{\Eh}}{\tau\|\bldv_h\|_{\Th}^2+2\mu|\!|\!|\underline{\bldv}_h|\!|\!|_{1,h}^2} \\
		&\lesssim
		\frac{1}{\lambda}\|\overline{p}_{h}\|_{\Th}^2 + \sup_{\linev_h\in\lineV_{h,0}^k}\frac{(\frac{1}{2\mu}\|\overline{p}_{h,1}\|_{\Th}^2+\frac{1}{\tau h}\|\jmp{\overline{p}_{h,2}}\|_{\Eh}^2)(2\mu\|\nabla\cdot\bldv_h\|^2_{\Th}+\tau\|\bldv_h\|^2_{\Th})}{\tau\|\bldv_h\|_{\Th}^2+2\mu|\!|\!|\underline{\bldv}_h|\!|\!|_{1,h}^2} \\
		&\leq
		\frac{1}{\lambda}\|\overline{p}_{h}\|_{\Th}^2 +\frac{1}{2\mu}\|\overline{p}_{h,1}\|_{\Th}^2+\frac{1}{\tau h}\|\jmp{\overline{p}_{h,2}}\|_{\Eh}^2.
\end{align*}
Since the splitting of $\overline{p}_{h}$ is arbitrary, we always have $\langle\mathbf{S}'\overline{\matrixP},\overline{\matrixP}\rangle_2
\lesssim
|\overline{p}_{h}|_\ast^2,\;\forall\overline{p}_{h}\in\overline{Q}_h$ and this completes the proof.
\end{enumerate}
\end{proof}

Next, we present the matrix representation of the newly defined norm $|\overline{p}_h|_\ast$ following the similar analysis procedure as in \cite[Theorem 2.5]{Olshanskii06}. It naturally follows from Theorem \ref{theo:SchurEquivNorm} that this matrix representation is spectrally equivalent to the Schur complement $\mathbf{S}'$ and its inverse is a robust Schur complement preconditioner with respect to model parameters and mesh size. We define matrix $\mathbf{M}$, $\mathbf{N}$ by \[\langle\mathbf{M}\overline{\matrixP},\overline{\matrixQ}\rangle_2=(\overline{p}_h,\overline{q}_h)_{\Th},\quad
\langle\mathbf{N}\overline{\matrixP},\overline{\matrixQ}\rangle_2=\frac{1}{h}\langle\jmp{\overline{p}_h},\jmp{\overline{q}_h}\rangle_{\Eh},\] for all $\overline{p}_h,\overline{q}_h\in\overline{Q}_h$ and we have the following result:
\begin{theorem}[Preconditioner for Schur Complement]
	\label{theo:SchurPre}
	Define the SPD matrix operator $\widetilde{\mathbf{S}'}$ on $\overline{Q}_h$ by $\widetilde{\mathbf{S}'}:=\frac{1}{\lambda}\mathbf{M}+\mathbf{M}(\tau\mathbf{M}+2\mu\mathbf{N})^{-1}\mathbf{N}$. We have:
	{
	\[
	|\overline{p}_h|_\ast^2
	=
	\langle\widetilde{\mathbf{S}'}\overline{\matrixP}_h,\overline{\matrixP}_h\rangle_2,
	\quad\forall\overline{p}_h\in\overline{Q}_h,
	\]
	and
	\[
	\kappa(\widetilde{\mathbf{S}'}^{-1}\mathbf{S}')\simeq 1.
	\]
	}
\end{theorem}
\begin{proof}
	In the definition of the norm $|\overline{p}_h|_{\ast}$ \eqref{pbarNorm}, by variational analysis the infimum is achieved when \[(\frac{1}{2\mu}\overline{p}_{h,1},\overline{q}_h)_{\Th}+\frac{1}{\tau h}\langle\jmp{\overline{p}_{h,1}-\overline{p}_h},\jmp{\overline{q}_h}\rangle_{\Eh}=0,\quad\forall\overline{q}_h\in\overline{Q}_h.\]
	By reformulating it into the matrix formulation, we get:
	\begin{align*}
		\left\langle\frac{1}{2\mu}\mathbf{M}\overline{\mathbf{P}}_{1}+\frac{1}{\tau}\mathbf{N}(\overline{\mathbf{P}}_1-\overline{\matrixP}),\overline{\matrixQ}\right\rangle_2
		=0, \\
		\overline{\matrixP}_1
		=
		\left(\frac{\tau}{2\mu}\mathbf{M}+\mathbf{N}\right)^{-1}\mathbf{N}\overline{\matrixP}.
	\end{align*}
	Therefore, combined with the fact that $(\frac{1}{2\mu}\overline{p}_{h,1},\overline{p}_{h,1}-\overline{p})_{\Th}+\frac{1}{\tau h}\langle\jmp{\overline{p}_{h,1}-\overline{p}_h},\jmp{\overline{p}_{h,1}-\overline{p}}\rangle_{\Eh}=0$, $|\overline{p}_h|_{\ast}^2$ can be explicitly expressed as
	\begin{align*}
		|\overline{p}_h|_{\ast}^2
		&=
		\frac{1}{\lambda}\|\overline{p}_h\|_{\Th}^2+\frac{1}{2\mu}\|\overline{p}_{h,1}\|_{\Th}^2+\frac{1}{\tau h}\|\jmp{\overline{p}_h-\overline{p}_{h,1}}\|_{\Eh}^2 \\
		&=
		\frac{1}{\lambda}(\overline{p}_h,\overline{p}_h)_{\Th}+\frac{1}{2\mu}(\overline{p}_h,\overline{p}_{h,1})_{\Th} \\
		&=
		\frac{1}{\lambda}\langle\mathbf{M}\overline{\mathbf{P}},\overline{\mathbf{P}}\rangle_2+\frac{1}{2\mu}\left\langle\mathbf{M}\left(\frac{\tau}{2\mu}\mathbf{M}+\mathbf{N}\right)^{-1}\mathbf{N}\overline{\mathbf{P}},\overline{\mathbf{P}}\right\rangle_2 \\
		&=
		\langle\widetilde{\mathbf{S}'}\overline{\matrixP},\overline{\matrixP}\rangle_2,
	\end{align*}
	where
	\begin{align*} 
		\widetilde{\mathbf{S}'}&=
		\frac{1}{\lambda}\mathbf{M}+\frac{1}{2\mu}\mathbf{M}\left(\frac{\tau}{2\mu}\mathbf{M}+\mathbf{N}\right)^{-1}\mathbf{N} \\
		&=
		\frac{1}{\lambda}\mathbf{M}+\mathbf{M}(\tau\mathbf{M}+2\mu\mathbf{N})^{-1}\mathbf{N}.
	\end{align*}
	{Then $\kappa(\widetilde{\mathbf{S}'}^{-1}\mathbf{S}')\simeq 1$ directly follows Theorem \ref{theo:SchurEquivNorm}}.
\end{proof}
}

\revise{
	Finally, we present an explicit and computable expression for $\widetilde{\mathbf{S}'}^{-1}$ which has a remarkably simple structure. At first glance, it seems quite difficult to exactly compute $\widetilde{\mathbf{S}'}^{-1}$ due to its complicated form when the parameters $\frac{1}{\lambda}$, $\mu$ and $\tau$ are not zero. This issue was partially addressed in our previous work\cite{fu2022monolithic} using an approximation to matrix inversion, where
	\begin{align}
		\label{LEpracS}
		\widetilde{\mathbf{S}'}^{-1} \approx
		\frac{2\mu\lambda}{2\mu+\lambda}(\mathbf{M})^{-1}+\tau(\frac{\tau}{\lambda}\mathbf{M}+\mathbf{N})^{-1}.
	\end{align}
	We note that a similar form to \eqref{LEpracS} can also be found in a recent paper\cite{olshanskii2022recycling} to precondition the Schur complement of a penalized surface incompressible fluid problem. Note that \eqref{LEpracS} is an exact inverse of $\widetilde{\mathbf{S}'}$ only if either of $\frac{1}{\lambda}$, $\tau$, or $\mu$ is zero, but not exact for the general case where $\frac{1}{\lambda}$, $\tau$, and $\mu$ are all not zero. 
	Here, we obtain a simple expression of the exact inverse of $\widetilde{\mathbf{S}'}$, which has a similar form as \eqref{LEpracS}, by exploring the structure of the matrix using the Woodbury matrix identity\cite{higham2002accuracy}.
	\begin{theorem}[Exact inverse of $\widetilde{\mathbf{S}'}$]
		\begin{align}
			\label{exactSinv}
			\widetilde{\mathbf{S}'}^{-1} = 
			\frac{2\mu\lambda}{2\mu + \lambda}\mathbf{M}^{-1}
			+ \tau\left(\frac{\lambda}{2\mu + \lambda}\right)^2\left(
				\frac{\tau}{2\mu + \lambda}\mathbf{M} + \mathbf{N}
			\right)^{-1}.
		\end{align}
	\end{theorem}
	\begin{proof}
		By algebraic manipulation, we have:
		\begin{align*}
			\widetilde{\mathbf{S}'}
			=&\frac{1}{\lambda}\mathbf{M}+\mathbf{M}\left(\tau\mathbf{M}+2\mu\mathbf{N}\right)^{-1}\mathbf{N}
			\\
			=&\frac{1}{\lambda}\mathbf{M}+\mathbf{M}\left(\tau\mathbf{M}+2\mu\mathbf{N}\right)^{-1}\frac{\tau\mathbf{M} + 2\mu\mathbf{N} - \tau\mathbf{M}}{2\mu}
			\\
			=&
			\frac{2\mu + \lambda}{2\mu\lambda}\mathbf{M} - \frac{\tau}{2\mu}\mathbf{M}\left(\tau\mathbf{M} + 2\mu\mathbf{N}\right)^{-1}\mathbf{M}.
		\end{align*}
		We quote Woodbury matrix identity\cite{higham2002accuracy} to get the inverses of sum matrices, which states:
		\begin{align*}
			\left(\mathbf{A} + \mathbf{U}\mathbf{C}\mathbf{V}\right)^{-1}
			= \mathbf{A}^{-1} - \mathbf{A}^{-1}\mathbf{U}\left(
				\mathbf{C}^{-1} + \mathbf{V}\mathbf{A}^{-1}\mathbf{U}
			\right)^{-1}\mathbf{V}\mathbf{A}^{-1},
		\end{align*}
		and the result follows by plugging into the above equation $\mathbf{A} = \frac{2\mu + \lambda}{2\mu\lambda}\mathbf{M}$, $\mathbf{U} = - \frac{\tau}{2\mu}\mathbf{M}$, $\mathbf{C} = \left(\tau\mathbf{M} + 2\mu\mathbf{N}\right)^{-1}$, and $\mathbf{V}= \mathbf{M}$.
	\end{proof}
}

\section{Numerical Results}
\label{sec:num}
In this section, we present two-dimensional and three-dimensional numerical experiments to verify the robustness of the proposed block-diagonal preconditioner \revise{$\mathrm{diag} [	\widetilde{\matrixA}_g^{-1},\;
\widetilde{\mathbf{S}'}^{-1}]$.
Here the ASP preconditioner $\widetilde{\matrixA}_g^{-1}$ is given in \eqref{asp} with $\mathbf{R}_g$ being the vertex-patch based block symmetric Gauss-Seidel smoother, and
$\bld B_0$ the approximate inverse of 
$\bld A_0$ using hypre's BoomerAMG \cite{Henson02},  and 
the Schur complement preconditioner 
$\widetilde{\mathbf{S}'}^{-1}$ is given in \eqref{exactSinv}, where the matrix inverses are again replaced by hypre's BoomerAMG approximations.}
 The first and second examples are based on the generalized Stokes equations, \revise{and the third and fourth examples are based on the steady and unsteady linear elasticity equations.} All results are obtained by using the NGSolve software \cite{Schoberl16}. All codes are available at \url{https://github.com/WZKuang/pc-hdg-saddle}.

The MINRES solver with relative tolerance of $10^{-8}$ is used to solve the condensed $H$(div)-conforming HDG scheme \eqref{condensedM}, starting with a random vector to ensure that the initial error is not smooth. 

\begin{table}[ht]
	\centering 
	{%
		\begin{tabular}{c||c|ccc|c|ccc}
			\toprule[1pt]
			& \multicolumn{4}{c|}{$\text{2D, }\mu=1$}
			& \multicolumn{4}{c}{$\text{3D, }\mu=1$}
			\\
			\midrule
			$k$
			& $\frac{1}{h}$  
			& $\tau=0$ & $\tau=1$ &$\tau=10^2$
			& $\frac{1}{h}$  
			& $\tau=0$ & $\tau=1$ &$\tau=10^2$
			\\
			\midrule
			\multirow{4}{2mm}{2} 
			&8& 57& 60& 54&4 & 71& 71&61
			\\
			&16& 58& 60& 56&8 & 76& 78&72
			\\
			&32& 57& 61& 57&12 & 79& 81&74
			\\
			&64& 58& 61& 58&16 & 77& 81&76
			\\
			\midrule
			\multirow{4}{2mm}{3} 
			&8& 60& 63& 57&4 & 74& 77&67
			\\
			&16& 61& 66& 60&8 & 77& 83&75
			\\
			&32& 61& 64& 61&12 & 77& 83&77
			\\
			&64& 61& 64& 61&16 & 77& 81&78
			\\
			\midrule
			\multirow{4}{2mm}{4} 
			&8& 64& 66& 59&4 & 80& 83&71
			\\
			&16& 65& 67& 62&8 & 82& 86&78
			\\
			&32& 63& 67& 63&12 & 80& 88&82
			\\
			&64& 62& 67& 64&16 & 80& 88&82
			\\
			\bottomrule[1pt]
	\end{tabular}}
	\vspace{2ex}  
	\caption{\it \textbf{Generalized Stokes in lid-driven cavity model.} 
		Iteration counts for the preconditioned MINRES solver.} 
	\label{table:i1} 
\end{table}

\begin{table}[ht]
	\centering 
	{%
		\begin{tabular}{c||c|ccc|c|ccc}
			\toprule[1pt]
			& \multicolumn{4}{c|}{$\text{2D, }\mu=1$}
			& \multicolumn{4}{c}{$\text{3D, }\mu=1$}
			\\
			\midrule
			$k$
			& $\frac{1}{h}$  
			& $\tau=0$ & $\tau=1$ &$\tau=10^2$
			& $\frac{1}{h}$  
			& $\tau=0$ & $\tau=1$ &$\tau=10^2$
			\\
			\midrule
			\multirow{4}{2mm}{2} 
			&8& 87& 80& 59&4 &135 &117 &81
			\\
			&16& 94& 84& 60&8 &159 &145 &109
			\\
			&32& 96& 87& 63&12 &164 &145 &110
			\\
			&64& 97& 87& 63&16 &158 &139 &108
			\\
			\midrule
			\multirow{4}{2mm}{3} 
			&8& 92& 83& 59&4 & 153& 131&90
			\\
			&16& 98& 91& 64&8 & 172& 153&112
			\\
			&32& 98& 90& 66&12 & 181& 157&116
			
			\\
			&64& 96& 88& 66&16 & 168& 151&108
			\\
			\midrule
			\multirow{4}{2mm}{4} 
			&8& 94& 86& 62&4 &161 & 141&94
			\\
			&16& 101& 92& 64&8 &186 & 164&115
			\\
			&32& 100& 92& 67&12 &193 & 166&119
			
			\\
			&64& 100& 91& 68&16 & 185& 153&114
			\\
			\bottomrule[1pt]
	\end{tabular}}
	\vspace{2ex}  
	\caption{\it \textbf{Generalized Stokes in backward-facing step flow model.} 
		Iteration counts for the preconditioned MINRES solver.} 
	\label{table:i2} 
\end{table}

\subsection{\revise{The generalized Stokes equations}}
\revise{The generalized Stokes problem fits into the general setting \eqref{model} with $\lambda=+\infty$ and it is easily verified that $\widetilde{\mathbf{S}'}^{-1}=2\mu\mathbf{M}^{-1}+\tau\mathbf{N}^{-1}$.} We choose the model problems of the lid-driven cavity and the backward-facing step flow as in \cite{farrell2019augmented}.

For the lid-driven cavity, we take the domain to be unit square/cube $\Omega_c=[0,1]^d$, where $d$ is the space dimension. An inhomogeneous Dirichlet boundary condition $\bld{u}=[4x(1-x),0]^\trans$ in 2D or $\bld{u}=[16x(1-x)y(1-y),0,0]^\trans$ in 3D is set on the top side, with no-slip boundary condition for all other sides. 
For the backward-facing step flow, we choose the domain $\Omega_b=([0.5,4]\times[0,0.5])\cup([0,4]\times[0.5,1])$ in 2D or $\Omega_b=\left(([0.5,4]\times[0,0.5])\cup([0,4]\times[0.5,1])\right)\times[0,1]$ in 3D. An inhomogeneous boundary condition $\bld{u}=[16(1-y)(y-0.5),0]^\trans$ in 2D or $\bld{u}=[64(1-y)(y-0.5)z(1-z),0,0]^\trans$ in 3D is set for the inlet velocity on $\{x=0\}$, with do-nothing boundary condition on $\{x=4\}$ and no-slip boundary condition on the remaining sides. 

In both model problems, the domains are divided into uniform simplicial meshes with mesh size $h_{start}$ followed by three-level refinement. $h_{start}$ is chosen to be 8 in 2D numerical experiments and 4 in 3D cases due to the limit of computation capability. The source function $\bld{f}$ is set to be $\bld{0}$. The value of $\mu$ is fixed to be 1 and $\tau$ is chosen from \{0,1,100\}. The change of iteration counts with the increase of polynomial degrees is also examined, with $k\in\{2,3,4\}$. The iteration counts of lid-driven cavity model are recorded in Table \ref{table:i1}, and iteration counts of backward-facing step flow model are in Table \ref{table:i2}.

As observed from Table \ref{table:i1}, the iteration counts of both two-dimensional and three-dimensional lid-driven cavity model problems are independent of mesh size $h$ for a fixed polynomial order $k$ and reaction parameter $\tau$. The iteration counts are also robust with respect to the value of $\tau$, decreasing as $\tau$ increases from 0 to 100, which is expected considering the velocity block becomes more similar to a mass matrix as $\tau$ increases. 
 
Moreover, we find the iteration count only increases very mildly as the polynomial degree increases from $k=2$ to $k=4$. 
It also needs to be noted that the iteration counts in three-dimensional cases are higher when compared to the 
two-dimensional counterparts. 
The results from Table \ref{table:i2} are similar to those from Table \ref{table:i1}, where it needs to be noted that the iteration counts in the backward-facing step flow problem are higher than those in the lid-driven cavity problem when other parameters are the same. Therefore, the iteration counts of our preconditioner are dependent on the shape of the domain, or more specifically aspect ratio of the domain.

\subsection{\revise{The steady and unsteady linear elasticity}}
\revise{
	For the steady and unsteady linear elasticity equations, we use the same domain and boundary conditions of the lid-driven cavity model problem in the generalized Stokes equations. The source term $\bld{f}$ is again set to be $\bld{0}$.
}

\revise{
	The steady linear elasticity equations fit into \eqref{model} with $\tau=0$, and the corresponding Schur complement preconditioner is $\widetilde{\mathbf{S}'}^{-1} = \frac{2\mu\lambda}{\lambda+2\mu}\mathbf{M}^{-1}$. The value of $\frac{1}{\lambda}$ is chosen from $\{10^{-4},10^{-1}, 1\}$, and all other settings are the same as the numerical experiments of the generalized Stokes equations. The iteration counts are recorded in Table \ref{table:i3}.
}	

\revise{
For the unsteady linear elasticity equations, we take $\mu=1$
and vary the value of $\tau$ from $\{10,\; 10^2,\; 10^3,\; 10^4\}$ and $\frac{1}{\lambda}$ from $\{10^{-4},\; 10^{-1},\; 1,\; 10\}$.
  The iteration counts 
  for $k=2$ with different mesh sizes are recorded in Table \ref{table:i4}, while those for $k=3$ are recorded in Table \ref{table:i5}.
}

\revise{
	Results from Table \ref{table:i3}--\ref{table:i5} verify the robustness of our 
 	the block-diagonal preconditioner with respect to mesh size and
 	model parameters,
 	with the nearly-incompressible cases taking more iterations than the compressible ones. 
}
\begin{table}[ht]
	\centering 
	{%
		\begin{tabular}{c||c|ccc|c|ccc}
			\toprule[1pt]
			& \multicolumn{4}{c|}{$\text{2D, }\mu=1$}
			& \multicolumn{4}{c}{$\text{3D, }\mu=1$}
			\\
			\midrule
			$k$
			& $\frac{1}{h}$  
			& $\frac{1}{\lambda}=10^{-4}$ & $\frac{1}{\lambda}=10^{-1}$ &$\frac{1}{\lambda}=1$
			& $\frac{1}{h}$  
			& $\frac{1}{\lambda}=10^{-4}$ & $\frac{1}{\lambda}=10^{-1}$ &$\frac{1}{\lambda}=1$
			\\
			\midrule
			\multirow{4}{2mm}{2} 
			&8& 89& 57& 38&4 & 111& 62&36
			\\
			&16& 90& 57& 38&8 & 80& 66&38
			\\
			&32& 61& 59& 38&12 & 81& 67&38
			\\
			&64& 61& 59& 37&16 & 81& 68&38
			\\
			\midrule
			\multirow{4}{2mm}{3} 
			&8& 96& 61& 39&4 & 119& 65&37
			\\
			&16& 73& 62& 39&8 & 83& 68&39
			\\
			&32& 66& 61& 39&12 & 84& 68&39
			\\
			&64& 66& 61& 39&16 & 82& 68&39
			\\
			\midrule
			\multirow{4}{2mm}{4} 
			&8& 91& 65& 44&4 & 125& 70& 39
			\\
			&16& 71& 66& 44&8 & 88& 72&42
			\\
			&32& 68& 66& 42&12 & 88& 72&41
			\\
			&64& 67& 66& 42&16 & 88& 72&41
			\\
			\bottomrule[1pt]
	\end{tabular}}
	\vspace{2ex}  
	\caption{\it \textbf{Steady linear elasticity in unit square/cube.} 
		Iteration counts for the preconditioned MINRES solver.} 
	\label{table:i3} 
\end{table}

\begin{table}[ht]
	{\revise{
    	\renewcommand{\arraystretch}{1.1} 
		\begin{tabular}{ccccccccccccc}
			\toprule[1pt]
			& \multicolumn{12}{c}{2D, $\mu = 1$} \\ \hline
		   \multicolumn{1}{c|}{} & \multicolumn{4}{c|}{$1/h = 32$} & \multicolumn{4}{c|}{$1/h = 64$} & \multicolumn{4}{c}{$1/h = 128$} \\ \hline
		   \multicolumn{1}{c|}{\multirow{2}{*}{$\tau$}} & \multicolumn{4}{c|}{$\lambda$} & \multicolumn{4}{c|}{$\lambda$} & \multicolumn{4}{c}{$\lambda$} \\
		   \multicolumn{1}{c|}{} & $10^{-4}$ & $10^{-1}$ & $1$ & \multicolumn{1}{c|}{$10$} & $10^{-4}$ & $10^{-1}$ & $1$ & \multicolumn{1}{c|}{$10$} & $10^{-4}$ & $10^{-1}$ & $1$ & $10$ \\ \hline
		   \multicolumn{1}{c|}{$10$} & 60 & 53 & 36 & \multicolumn{1}{c|}{29} & 60 & 53 & 36 & \multicolumn{1}{c|}{27} & 60 & 53 & 36 & 27 \\
		   \multicolumn{1}{c|}{$10^2$} & 59 & 52 & 38 & \multicolumn{1}{c|}{28} & 59 & 53 & 36 & \multicolumn{1}{c|}{27} & 60 & 53 & 36 & 27 \\
		   \multicolumn{1}{c|}{$10^3$} & 57 & 51 & 35 & \multicolumn{1}{c|}{27} & 58 & 54 & 37 & \multicolumn{1}{c|}{27} & 60 & 54 & 37 & 27 \\
		   \multicolumn{1}{c|}{$10^4$} & 50 & 42 & 30 & \multicolumn{1}{c|}{23} & 56 & 48 & 34 & \multicolumn{1}{c|}{26} & 59 & 51 & 36 & 27 \\ 
		   \midrule[1pt]\midrule[1pt]
			& \multicolumn{12}{c}{3D, $\mu = 1$} \\ \hline
		   \multicolumn{1}{c|}{} & \multicolumn{4}{c|}{$1/h = 8$} & \multicolumn{4}{c|}{$1/h = 12$} & \multicolumn{4}{c}{$1/h = 16$} \\ \hline
		   \multicolumn{1}{c|}{$\tau$} & \multicolumn{4}{c|}{$\lambda$} & \multicolumn{4}{c|}{$\lambda$} & \multicolumn{4}{c}{$\lambda$} \\
		   \multicolumn{1}{l|}{} & $10^{-4}$ & $10^{-1}$ & $1$ & \multicolumn{1}{c|}{$10$} & $10^{-4}$ & $10^{-1}$ & $1$ & \multicolumn{1}{c|}{$10$} & $10^{-4}$ & $10^{-1}$ & $1$ & $10$ \\ \hline
		   \multicolumn{1}{c|}{$10$} & 76 & 59 & 37 & \multicolumn{1}{c|}{28} & 79 & 60 & 38 & \multicolumn{1}{c|}{28} & 80 & 60 & 38 & 28 \\
		   \multicolumn{1}{c|}{$10^2$} & 72 & 56 & 34 & \multicolumn{1}{c|}{25} & 76 & 60 & 37 & \multicolumn{1}{c|}{26} & 77 & 60 & 38 & 26 \\
		   \multicolumn{1}{c|}{$10^3$} & 60 & 47 & 28 & \multicolumn{1}{c|}{21} & 66 & 51 & 31 & \multicolumn{1}{c|}{22} & 71 & 53 & 33 & 23 \\
		   \multicolumn{1}{c|}{$10^4$} & 49 & 29 & 18 & \multicolumn{1}{c|}{14} & 52 & 36 & 22 & \multicolumn{1}{c|}{16} & 56 & 40 & 24 & 17\\
		   \bottomrule[1pt]
		   \end{tabular}
		   \vspace{2ex}  
		   \caption{\it \textbf{Unsteady linear elasticity in unit square/cube.} 
		   Iteration counts for the preconditioned MINRES solver.
		   Polynomial degree $k=2$.
		   } 
	   	   \label{table:i4} 
	}}
\end{table}

\begin{table}[ht]
	{\revise{
    	\renewcommand{\arraystretch}{1.1} 
		\begin{tabular}{ccccccccccccc}
		    \toprule[1pt]
             & \multicolumn{12}{c}{2D, $\mu = 1$} \\ \hline
            \multicolumn{1}{c|}{} & \multicolumn{4}{c|}{$1/h = 32$} & \multicolumn{4}{c|}{$1/h = 64$} & \multicolumn{4}{c}{$1/h = 128$} \\ \hline
            \multicolumn{1}{c|}{\multirow{2}{*}{$\tau$}} & \multicolumn{4}{c|}{$\lambda$} & \multicolumn{4}{c|}{$\lambda$} & \multicolumn{4}{c}{$\lambda$} \\
            \multicolumn{1}{c|}{} & $10^{-4}$ & $10^{-1}$ & $1$ & \multicolumn{1}{c|}{$10$} & $10^{-4}$ & $10^{-1}$ & $1$ & \multicolumn{1}{c|}{$10$} & $10^{-4}$ & $10^{-1}$ & $1$ & $10$ \\ \hline
            \multicolumn{1}{c|}{$10$} & 63 & 55 & 37 & \multicolumn{1}{c|}{30} & 63 & 55 & 37 & \multicolumn{1}{c|}{30} & 65 & 55 & 37 & 28 \\
            \multicolumn{1}{c|}{$10^2$} & 62 & 54 & 39 & \multicolumn{1}{c|}{30} & 63 & 55 & 39 & \multicolumn{1}{c|}{30} & 63 & 55 & 39 & 28 \\
            \multicolumn{1}{c|}{$10^3$} & 62 & 53 & 36 & \multicolumn{1}{c|}{27} & 65 & 56 & 39 & \multicolumn{1}{c|}{28} & 64 & 56 & 39 & 28 \\
            \multicolumn{1}{c|}{$10^4$} & 53 & 45 & 32 & \multicolumn{1}{c|}{25} & 60 & 50 & 35 & \multicolumn{1}{c|}{27} & 62 & 53 & 36 & 27 \\
            \midrule[1pt]\midrule[1pt]
             & \multicolumn{12}{c}{3D, $\mu = 1$} \\ \hline
            \multicolumn{1}{c|}{} & \multicolumn{4}{c|}{$1/h = 8$} & \multicolumn{4}{c|}{$1/h = 12$} & \multicolumn{4}{c}{$1/h = 16$} \\ \hline
            \multicolumn{1}{c|}{$\tau$} & \multicolumn{4}{c|}{$\lambda$} & \multicolumn{4}{c|}{$\lambda$} & \multicolumn{4}{c}{$\lambda$} \\
            \multicolumn{1}{l|}{} & $10^{-4}$ & $10^{-1}$ & $1$ & \multicolumn{1}{c|}{$10$} & $10^{-4}$ & $10^{-1}$ & $1$ & \multicolumn{1}{c|}{$10$} & $10^{-4}$ & $10^{-1}$ & $1$ & $10$ \\ \hline
            \multicolumn{1}{c|}{$10$} & 79 & 60 & 38 & \multicolumn{1}{c|}{28} & 82 & 61 & 39 & \multicolumn{1}{c|}{29} & 82 & 61 & 39 & 27 \\
            \multicolumn{1}{c|}{$10^2$} & 75 & 58 & 37 & \multicolumn{1}{c|}{26} & 77 & 59 & 38 & \multicolumn{1}{c|}{27} & 80 & 60 & 38 & 27 \\
            \multicolumn{1}{c|}{$10^3$} & 69 & 50 & 31 & \multicolumn{1}{c|}{23} & 73 & 55 & 33 & \multicolumn{1}{c|}{25} & 76 & 56 & 35 & 23 \\
            \multicolumn{1}{c|}{$10^4$} & 53 & 30 & 19 & \multicolumn{1}{c|}{15} & 58 & 38 & 24 & \multicolumn{1}{c|}{18} & 62 & 44 & 27 & 19
            \\
            \bottomrule[1pt]
        \end{tabular}
		   \vspace{2ex}  
		   \caption{\it \textbf{Unsteady linear elasticity in unit square/cube.} 
		   Iteration counts for the preconditioned MINRES solver.
		   Polynomial degree $k=3$.
		   } 
	   	   \label{table:i5} 
	}}
\end{table}

\section{Conclusion}
\label{sec:conclude}
In this paper, we presented a robust block-diagonal preconditioner with respect to mesh size $h$ and model parameters for the condensed $H$(div)-conforming HDG schemes for the parameter-dependent saddle point problems, including the generalized Stokes equations and the linear elasticity equations. For the stiffness matrix, we extended from the optimal ASP for the $H$(div)-conforming HDG scheme for the reaction-diffusion equations, which was developed in our previous study. \revise{For the Schur complement, we obtained a general matrix formulation spectrally equivalent to the Schur complement in Theorem \ref{theo:SchurPre}, based on a newly defined parameter-dependent norm on the element-wise constant space. Then an explicit computable exact inverse is obtained via the Woodbury matrix identity.
Numerical results verify the robustness of the proposed block preconditioner in both two- and three-dimensions.}

\textbf{Acknowledgement:}
The authors would like to thank two anonymous reviewers for constructive criticism, which enables a better presentation of the material in this paper.

\bibliography{pc-hdg-saddle}

\begin{thebibliography}{10}

\bibitem{AinsworthFu18}
{\sc M.~Ainsworth and G.~Fu}, {\em Fully computable a posteriori error bounds
  for hybridizable discontinuous {G}alerkin finite element approximations}, J.
  Sci. Comput., 77 (2018), pp.~443--466.

\bibitem{barrenechea2019hybrid}
{\sc G.~R. Barrenechea, M.~Bosy, V.~Dolean, F.~Nataf, and P.-H. Tournier}, {\em
  Hybrid discontinuous galerkin discretisation and domain decomposition
  preconditioners for the stokes problem}, Computational Methods in Applied
  Mathematics, 19 (2019), pp.~703--722.

\bibitem{benzi2005numerical}
{\sc M.~Benzi, G.~H. Golub, and J.~Liesen}, {\em Numerical solution of saddle
  point problems}, Acta numerica, 14 (2005), pp.~1--137.

\bibitem{benzi2008some}
{\sc M.~Benzi and A.~J. Wathen}, {\em Some preconditioning techniques for
  saddle point problems}, in Model order reduction: theory, research aspects
  and applications, Springer, 2008, pp.~195--211.

\bibitem{Bergh:1617905}
{\sc J.~Bergh and J.~Löfström}, {\em {Interpolation spaces: an
  introduction}}, Grundlehren der mathematischen Wissenschaften A Series of
  Comprehensive Studies in Mathematics, Springer, Berlin, 1976.

\bibitem{betteridge2021multigrid}
{\sc J.~Betteridge, T.~H. Gibson, I.~G. Graham, and E.~H. M{\"u}ller}, {\em
  Multigrid preconditioners for the hybridised discontinuous galerkin
  discretisation of the shallow water equations}, Journal of Computational
  Physics, 426 (2021), p.~109948.

\bibitem{bramble1997iterative}
{\sc J.~H. Bramble and J.~E. Pasciak}, {\em Iterative techniques for time
  dependent stokes problems}, Computers \& Mathematics with Applications, 33
  (1997), pp.~13--30.

\bibitem{brenner2004korn}
{\sc S.~C. Brenner}, {\em Korn's inequalities for piecewise h1 vector fields},
  Mathematics of Computation,  (2004), pp.~1067--1087.

\bibitem{cahouet1988some}
{\sc J.~Cahouet and J.-P. Chabard}, {\em Some fast 3d finite element solvers
  for the generalized stokes problem}, International Journal for Numerical
  Methods in Fluids, 8 (1988), pp.~869--895.

\bibitem{chen2014robust}
{\sc H.~Chen, P.~Lu, and X.~Xu}, {\em A robust multilevel method for
  hybridizable discontinuous galerkin method for the helmholtz equation},
  Journal of Computational Physics, 264 (2014), pp.~133--151.

\bibitem{Cockburn16}
{\sc B.~Cockburn}, {\em Static condensation, hybridization, and the devising of
  the {HDG} methods}, in Building bridges: connections and challenges in modern
  approaches to numerical partial differential equations, vol.~114 of Lect.
  Notes Comput. Sci. Eng., Springer, [Cham], 2016, pp.~129--177.

\bibitem{cockburn2018discontinuous}
{\sc B.~Cockburn}, {\em Discontinuous galerkin methods for computational fluid
  dynamics}, Encyclopedia of Computational Mechanics Second Edition,  (2018),
  pp.~1--63.

\bibitem{CockburnDubois14}
{\sc B.~Cockburn, O.~Dubois, J.~Gopalakrishnan, and S.~Tan}, {\em Multigrid for
  an {HDG} method}, IMA J. Numer. Anal., 34 (2014), pp.~1386--1425.

\bibitem{CockburnGopalakrishnanLazarov09}
{\sc B.~Cockburn, J.~Gopalakrishnan, and R.~Lazarov}, {\em Unified
  hybridization of discontinuous {Galerkin}, mixed and continuous {Galerkin}
  methods for second order elliptic problems}, SIAM J. Numer. Anal., 47 (2009),
  pp.~1319--1365.

\bibitem{cockburn2016hdg}
{\sc B.~Cockburn, N.~C. Nguyen, and J.~Peraire}, {\em Hdg methods for
  hyperbolic problems}, in Handbook of Numerical Analysis, vol.~17, Elsevier,
  2016, pp.~173--197.

\bibitem{fabien2019manycore}
{\sc M.~S. Fabien, M.~G. Knepley, R.~T. Mills, and B.~M. Rivi{\`e}re}, {\em
  Manycore parallel computing for a hybridizable discontinuous galerkin nested
  multigrid method}, SIAM Journal on Scientific Computing, 41 (2019),
  pp.~C73--C96.

\bibitem{farrell2019augmented}
{\sc P.~E. Farrell, L.~Mitchell, and F.~Wechsung}, {\em An augmented lagrangian
  preconditioner for the 3d stationary incompressible navier--stokes equations
  at high reynolds number}, SIAM Journal on Scientific Computing, 41 (2019),
  pp.~A3073--A3096.

\bibitem{fu2021uniform}
{\sc G.~Fu}, {\em Uniform auxiliary space preconditioning for hdg methods for
  elliptic operators with a parameter dependent low order term}, SIAM Journal
  on Scientific Computing, 43 (2021), pp.~A3912--A3937.

\bibitem{fu2019parameter}
{\sc G.~Fu, Y.~Jin, and W.~Qiu}, {\em Parameter-free superconvergent h
  (div)-conforming hdg methods for the brinkman equations}, IMA Journal of
  Numerical Analysis, 39 (2019), pp.~957--982.

\bibitem{fu2022monolithic}
{\sc G.~Fu and W.~Kuang}, {\em A monolithic divergence-conforming hdg scheme
  for a linear fluid-structure interaction model}, SIAM Journal on Numerical
  Analysis, 60 (2022), pp.~631--658.

\bibitem{fu2021locking}
{\sc G.~Fu, C.~Lehrenfeld, A.~Linke, and T.~Streckenbach}, {\em Locking free
  and gradient robust h (div)-conforming hdg methods for linear elasticity},
  Journal of Scientific Computing, 86 (2021), pp.~1--30.

\bibitem{gander2018analysis}
{\sc M.~Gander and S.~Hajian}, {\em Analysis of schwarz methods for a
  hybridizable discontinuous galerkin discretization: the many-subdomain case},
  Mathematics of Computation, 87 (2018), pp.~1635--1657.

\bibitem{gander2015analysis}
{\sc M.~J. Gander and S.~Hajian}, {\em Analysis of schwarz methods for a
  hybridizable discontinuous galerkin discretization}, SIAM Journal on
  Numerical Analysis, 53 (2015), pp.~573--597.

\bibitem{he2016optimized}
{\sc Y.-X. He, L.~Li, S.~Lanteri, and T.-Z. Huang}, {\em Optimized schwarz
  algorithms for solving time-harmonic maxwell’s equations discretized by a
  hybridizable discontinuous galerkin method}, Computer Physics Communications,
  200 (2016), pp.~176--181.

\bibitem{Henson02}
{\sc V.~E. Henson and U.~M. Yang}, {\em Boomer{AMG}: a parallel algebraic
  multigrid solver and preconditioner}, vol.~41, 2002, pp.~155--177.
\newblock Developments and trends in iterative methods for large systems of
  equations---in memoriam R\"{u}diger Weiss (Lausanne, 2000).

\bibitem{higham2002accuracy}
{\sc N.~J. Higham}, {\em Accuracy and stability of numerical algorithms}, SIAM,
  2002.

\bibitem{johnson1985matrix}
{\sc C.~R. Johnson and R.~A. Horn}, {\em Matrix analysis}, Cambridge university
  press Cambridge, 1985.

\bibitem{kobelkov2000effective}
{\sc G.~M. Kobelkov and M.~A. Olshanskii}, {\em Effective preconditioning of
  uzawa type schemes for a generalized stokes problem}, Numerische Mathematik,
  86 (2000), pp.~443--470.

\bibitem{Lehrenfeld10}
{\sc C.~Lehrenfeld}, {\em Hybrid {Discontinuous} {Galerkin} methods for solving
  incompressible flow problems}.
\newblock Diploma Thesis, MathCCES/IGPM, RWTH Aachen, 2010.

\bibitem{LehrenfeldSchoberl16}
{\sc C.~Lehrenfeld and J.~Sch\"{o}berl}, {\em High order exactly
  divergence-free hybrid discontinuous {G}alerkin methods for unsteady
  incompressible flows}, Comput. Methods Appl. Mech. Engrg., 307 (2016),
  pp.~339--361.

\bibitem{li2014hybridizable}
{\sc L.~Li, S.~Lanteri, and R.~Perrussel}, {\em A hybridizable discontinuous
  galerkin method combined to a schwarz algorithm for the solution of 3d
  time-harmonic maxwell's equation}, Journal of Computational Physics, 256
  (2014), pp.~563--581.

\bibitem{lu2020hmg}
{\sc P.~Lu, A.~Rupp, and G.~Kanschat}, {\em Hmg--homogeneous multigrid for
  hdg}, arXiv preprint arXiv:2011.14018,  (2020).

\bibitem{Mardal04}
{\sc K.-A. Mardal and R.~Winther}, {\em Uniform preconditioners for the time
  dependent {S}tokes problem}, Numer. Math., 98 (2004), pp.~305--327.

\bibitem{mardal2011preconditioning}
{\sc K.-A. Mardal and R.~Winther}, {\em Preconditioning discretizations of
  systems of partial differential equations}, Numerical Linear Algebra with
  Applications, 18 (2011), pp.~1--40.

\bibitem{muralikrishnan2017ihdg}
{\sc S.~Muralikrishnan, M.-B. Tran, and T.~Bui-Thanh}, {\em ihdg: An iterative
  hdg framework for partial differential equations}, SIAM Journal on Scientific
  Computing, 39 (2017), pp.~S782--S808.

\bibitem{muralikrishnan2018improved}
\leavevmode\vrule height 2pt depth -1.6pt width 23pt, {\em An improved
  iterative hdg approach for partial differential equations}, Journal of
  Computational Physics, 367 (2018), pp.~295--321.

\bibitem{nguyen2012hybridizable}
{\sc N.~C. Nguyen and J.~Peraire}, {\em Hybridizable discontinuous galerkin
  methods for partial differential equations in continuum mechanics}, Journal
  of Computational Physics, 231 (2012), pp.~5955--5988.

\bibitem{NguyenPeraireCockburnConvDiff09}
{\sc N.~C. Nguyen, J.~Peraire, and B.~Cockburn}, {\em An implicit high-order
  hybridizable discontinuous {G}alerkin method for linear convection-diffusion
  equations}, J. Comput. Phys., 228 (2009), pp.~3232--3254.

\bibitem{Olshanskii06}
{\sc M.~A. Olshanskii, J.~Peters, and A.~Reusken}, {\em Uniform preconditioners
  for a parameter dependent saddle point problem with application to
  generalized {S}tokes interface equations}, Numer. Math., 105 (2006),
  pp.~159--191.

\bibitem{olshanskii2022recycling}
{\sc M.~A. Olshanskii and A.~Zhiliakov}, {\em Recycling augmented lagrangian
  preconditioner in an incompressible fluid solver}, Numerical Linear Algebra
  with Applications, 29 (2022), p.~e2415.

\bibitem{pestana2015natural}
{\sc J.~Pestana and A.~J. Wathen}, {\em Natural preconditioning and iterative
  methods for saddle point systems}, siam REVIEW, 57 (2015), pp.~71--91.

\bibitem{peters2005fast}
{\sc J.~Peters, V.~Reichelt, and A.~Reusken}, {\em Fast iterative solvers for
  discrete stokes equations}, SIAM journal on scientific computing, 27 (2005),
  pp.~646--666.

\bibitem{qiu2016superconvergent}
{\sc W.~Qiu and K.~Shi}, {\em A superconvergent hdg method for the
  incompressible navier--stokes equations on general polyhedral meshes}, IMA
  Journal of Numerical Analysis, 36 (2016), pp.~1943--1967.

\bibitem{rhebergen2018preconditioning}
{\sc S.~Rhebergen and G.~N. Wells}, {\em Preconditioning of a hybridized
  discontinuous galerkin finite element method for the stokes equations},
  Journal of Scientific Computing, 77 (2018), pp.~1936--1952.

\bibitem{rhebergen2021preconditioning}
\leavevmode\vrule height 2pt depth -1.6pt width 23pt, {\em Preconditioning for
  a pressure-robust hdg discretization of the stokes equations}, arXiv preprint
  arXiv:2105.09152,  (2021).

\bibitem{Schoberl16}
{\sc J.~Sch{\"o}berl}, {\em {C}++11 {I}mplementation of {F}inite {E}lements in
  {NGS}olve}, 2014.
\newblock {ASC Report 30/2014, Institute for Analysis and Scientific Computing,
  Vienna University of Technology}.

\bibitem{schoberl2013domain}
{\sc J.~Sch{\"o}berl and C.~Lehrenfeld}, {\em Domain decomposition
  preconditioning for high order hybrid discontinuous galerkin methods on
  tetrahedral meshes}, in Advanced finite element methods and applications,
  Springer, 2013, pp.~27--56.

\bibitem{silvester1994fast}
{\sc D.~Silvester and A.~Wathen}, {\em Fast iterative solution of stabilised
  stokes systems part ii: Using general block preconditioners}, SIAM Journal on
  Numerical Analysis, 31 (1994), pp.~1352--1367.

\bibitem{tu2020analysis}
{\sc X.~Tu, B.~Wang, and J.~Zhang}, {\em Analysis of bddc algorithms for stokes
  problems with hybridizable discontinuous galerkin discretizations},  (2020).

\bibitem{wathen1993fast}
{\sc A.~Wathen and D.~Silvester}, {\em Fast iterative solution of stabilised
  stokes systems. part i: Using simple diagonal preconditioners}, SIAM Journal
  on Numerical Analysis, 30 (1993), pp.~630--649.

\bibitem{Xu96}
{\sc J.~Xu}, {\em The auxiliary space method and optimal multigrid
  preconditioning techniques for unstructured grids}, vol.~56, 1996,
  pp.~215--235.
\newblock International GAMM-Workshop on Multi-level Methods (Meisdorf, 1994).

\end{thebibliography}
\bibliographystyle{siam}

\end{document}